\documentclass[twoside,12pt]{article}
\usepackage[all, 2cell, dvips]{xy}
\usepackage{latexsym,amsmath, amsfonts, graphics, epsf, epic}
\usepackage{amsthm }
\usepackage{amssymb}
\usepackage{amsopn}
\usepackage{amscd}
\usepackage{color}

\theoremstyle{plain}
\newtheorem{theorem}{Theorem}[section]
\newtheorem{corollary}[theorem]{Corollary}
\newtheorem{lemma}[theorem]{Lemma}
\newtheorem{proposition}[theorem]{Proposition}

\theoremstyle{definition}

\newtheorem{remark}[theorem]{Remark}
\newtheorem{example}[theorem]{Example}
\newtheorem{definition}[theorem]{Definition}

\usepackage{amssymb}

\def\dim{\mbox{\rm dim }}

\def\l.l.o.{\it l.l.o}

\def\chiup{\raise 2pt\hbox{$\chi$}}

\title{The geometry of Polynomial Identities}
\author{ Claudio Procesi}

\begin{document}

\maketitle
  %\hskip9cm{\it Dedicated to J.P. Serre 90th birthday}
\tableofcontents

\newpage

\section{Introduction}\label{introduction}

In the last 50 years, or more,  several papers have dealt with connections between non--commutative algebra and algebraic geometry.

Sometimes the term {\em non--commutative  algebraic geometry} has been used, which I find misleading since there are several quite disjoint instances of these connections.  Loosely the only unifying ideas are some information on various sorts of {\em spectral theory} of operators.\smallskip

In this paper we want to  examine a particular theory, that of algebras with polynomial identities, and show  how a sequence of varieties of semi--simple representations, see \ref{vase},   appears naturally associated to a PI theory, see Theorems \ref{finKs} and \ref{intrisK}.

These varieties are among a series of discrete, geometric and combinatorial invariants of PI theories, together with these one should also consider some {\em moduli spaces}, which parametrise certain coherent sheaves on these representation varieties. At the moment these objects seem to be hard to treat although some hints will be given in the last part of the paper.\smallskip

In this paper {\em algebra} will always mean associative algebra over some field $F$.
We will also assume that $F$ is of characteristic 0, otherwise too many difficult and sometimes unsolved problems arise, when convenient we shall also assume that $F$ is algebraically closed.\smallskip

Let us recall the main definitions and refer the reader to Rowen's book for a comprehensive discussion of the basic material, \cite{rowen1} or V. Drensky, \cite{Drens2}   and V. Drensky--E. Formanek
\cite{DF}. \medskip

We start from the free associative algebra $F\langle X \rangle  $, in some finite or countable set of variables $x_i$, with basis the {\em words} in these variables, and whose elements  we think of as {\em non commutative polynomials}.  

Given any algebra $A$  we have for every map, $ X\to A,\  x_i\mapsto a_i$ an associated {\em evaluation} of   polynomials $f(x_1,\ldots,x_m)\mapsto f(a_1,\ldots,a_m)\in A$ and then
\begin{definition}\label{PI}\begin{itemize}
\item A polynomial $f(x_1,\ldots,x_m)\in F \langle  X \rangle  $  is a polynomial identity of $A$, short PI, if   $f(a_1,\ldots,a_m)=0$ for all evaluations in $A$.

\item A PI algebra is an algebra $A$ which satisfies a non--zero polynomial identity.

\item Two algebras $A,B$ are PI--equivalent if they satisfy the same polynomial identities.
\end{itemize}
\end{definition}

The set $Id(A)$ of PI's of $A$ in the variables $X$ is an ideal of  $F \langle  X \rangle  $ closed under all possible endomorphisms of $F \langle  X \rangle  $, that is substitutions of the variables $x_i$ with polynomials.  Such an ideal $I$ is called {\em a $T$--ideal,} the quotient $F \langle  X \rangle  /I$ satisfies all the identities in $I$. In fact $F \langle  X \rangle  /Id(A)$ is a free algebra in the category of all algebras which satisfy the PI's of $A$.

When $X$ is countable and  $I$ is   a $T$--ideal, then $I$ is the ideal of all polynomial identities of $F \langle  X \rangle  /I$.

One usually says that $F \langle  X \rangle  /I$ is a {\em relatively free algebra}.  \smallskip

\begin{remark}\label{anno}
There is a small annoying technical point,  sometimes it is necessary to work with algebras without 1 (in the sense that either there is truly not 1 or we do not consider it as part of the axioms). In this case the free algebra has to be taken without 1. This implies a  different notion of $T$--ideals, since in the case of algebras with 1, a $T$--ideal is also stable under specialising a variable to 1, while this is not allowed in the other case.  We leave to the reader to understand in which setting we are working.
\end{remark}\medskip

As usual {\em polynomials} have a dual aspect of symbolic expressions or functions, the same happens in PI theory where 
$F \langle  X \rangle  /I$ can be identified to an algebra of polynomial maps  $A^X\to A$ which commute with the action of the automorphism group of $A$.

Its  algebraic combinatorial nature is quite complex and usually impossible to describe in detail even for very special algebras $A$. In general the  algebra of all  polynomial maps commuting  with the action of the automorphism group of $A$ is strictly larger. A   major example is when $A=M_n(F)$  is the algebra of $n\times n$ matrices over $F$ where one quickly finds a connection with classical invariant theory.  Even this basic example cannot be described in full except in the trivial case $n=1$ and the  non--trivial  case $n=2$ cf. \cite{procesi3}.\medskip

 A major difficulty in the Theory is the fact that relatively free algebras $F \langle  X \rangle  /I$, even when $X$ is finite,  almost never satisfy the Noetherian condition, cf. Amitsur \cite{amitsur3}, that is the fact that ideals (right, left or bilateral)  have a finite set of generators.  This is not just a technical point but reflects some deep  combinatorics appearing.   This is in part overcome by the solution of the Specht problem by Kemer, Theorem \ref{Spe}.  Section \S \ref{KeTe} is devoted to a quick overview of the required Kemer theory. In particular we shall stress the role of {\em fundamental algebras}, see Definition \ref{FAL1}  and Theorem \ref{primoK},  which are the building pieces of PI equivalence classes of finite dimensional algebras.\smallskip
 
 In \S \ref{KeTe1} we start to draw some consequences of Kemer's theory. The first goal of this paper is to show, Theorem  \ref{finKs}, that, when $X$ is finite,  for a given relatively free algebra   $F \langle  X \rangle  /I$ there is a canonical finite filtration    by $T$--ideals  $K_i$ such that the quotients $K_i/K_{i-1}$ have natural structures of finitely generated modules over special finitely generated commutative algebras $\mathcal T_i$. In fact each piece $K_i/K_{i-1}$ is a two sided ideal  in some special {\em trace algebras }  which are  finitely generated modules over these commutative algebras. These algebras $\mathcal T_i$ are coordinate rings of certain representation varieties, \S \ref{vase} and  the word {\em geometry} appearing in the title refers to the  geometric description of the algebraic varieties  supporting the  various modules $K_i/K_{i-1}$,  see Theorem    \ref{intrisK}.\smallskip
 
 As next step we show that these varieties are natural quotient varieties  parametrising semi--simple representations.  As an important consequence in Corollary \ref{intrisK1}  we show that:
 
  {\em if two  fundamental algebras are PI equivalent then they have the same   semisimple part.}\smallskip

 Then from some explicit information on these varieties we shall deduce a number of corollaries  computing the dimension  of the relatively free algebras, Proposition \ref{dimfua} and  the  growth of the cocharacter sequence, see Definition \ref{coch}, using some general facts of invariant theory,  Corollary \ref{raUU}.\smallskip
 
In the final section \S \ref{modA}     we shall indicate a method  to {\em classify}  finite dimensional algebras up to PI--equivalence \S \ref{model}.
\section{Growth and the theory of Kemer\label{KeTe}}{\em In this section we recall basic results of PI theory  which motivate our research and are needed for the material of this paper}\subsubsection{Growth} Given $n$, we let $V_n$ denote the space of multilinear polynomials of degree $n$ in the variables $x_1,\ldots, x_n$. 
The space  $V_n$ has as basis the monomials  $x_{\sigma(1)}\ldots x_{\sigma(n)}$ as $\sigma$ runs over all permutations of $1,2,\ldots,n$,
so   $\dim V_n = n!$.

Identities can always be multilinearized, hence the subset $Id(A)\cap V_n$ plays a special role and,  in characteristic zero,
the ideal $Id(A)$ is completely determined by the sequence of multilinear identities $\{Id(A)\cap V_n\}_{n\ge 1}$. In order to
study  $\dim (Id(A)\cap V_n)$  we introduce the quotient space $V_n/(Id(A)\cap V_n)$ and its dimension
$$
c_n(A):=\dim\left (\frac{V_n}{Id(A)\cap V_n}  \right).
$$
The integer $c_n(A)$ is the $n$-th {\it codimension} of $A$. Clearly $c_n(A)$ determines
$\dim (Id(A)\cap V_n)$ since $\dim V_n$ is known.

The study of growth for PI algebra $A$ is mostly the study of the rate of growth of the sequence
$c_n(A)$ of its codimensions, as $n$ goes to infinity. For a full survey we refer to \cite{regev2}.  We have the following basic property proved by Regev.

\begin{theorem}~\cite{regev1}
  $c_n(A)$ is always exponentially bounded.
\end{theorem}
We have then the integrality theorem of Giambruno--Zaicev.

\begin{theorem}\label{giambruno-zaicev}\cite{giambruno.zaicev2}
Let $A$ be a  PI  algebra over a filed $F$ with $char(F)=0$, then the limit
$$
\lim_{n\to\infty}  c_n(A) ^{1/n}\in\mathbb N
$$
exists and is an integer, called {\em exponent}.
 \end{theorem}

The space $V_n/(Id(A)\cap V_n)$ is a representation of the symmetric group $S_n$ acting by permuting variables and \begin{definition}\label{coch}
The $S_n$ character of that space, $\chi_{S_n}(V_n/(Id(A)\cap V_n))$ is
denoted
$$
\chi_n(A)=\chi_{S_n}\left (\frac{V_n}{Id(A)\cap V_n}  \right),
$$
and is called the {\em $n$-th cocharacter} of $A$.
\end{definition} Since $c_n(A)=\deg \chi_n(A)$, cocharacters are refinement
of codimensions, and are important tool in their study. By a theorem of Amitsur--Regev and of Kemer, $\chi_n(A)$
is supported on some $(k,\ell)$ hook.
Shirshov's Height Theorem, \cite{shirshov2}, then implies that the multiplicities of the irreducible characters, in the cocharacter sequence, are polynomially bounded.

One of the goals of this paper is to discuss, see \S \ref{diref} and  \S \ref{icocc},  some further informations one can gather on these numbers using geometric methods.
 
\subsection{Three fundamental theorems}

We need to review the results and some of the techniques of the Theory of Kemer, presented in the monograph \cite{kemer}, see also \cite{AbK} or the forthcoming book  \cite{AGPR}.

A fundamental Theorem of Kemer   states that 
\begin{theorem}\label{Kem}[Kemer]
If $X$ is a finite set, a non--zero $T$--ideal $I$  of $F \langle  X \rangle  $ is the ideal of polynomial identities in the variables $X$ for a finite dimensional algebra $A$.
\end{theorem}   In other words any finitely generated PI algebra is PI equivalent to a finite dimensional algebra.

This is in fact the first part of a more general statement, let us consider the Grassmann algebra, thought of as super--algebra,  in countably many odd generators  $G:=\bigwedge[e_1,e_2,\ldots]$ decomposed as $G=G_0\oplus G_1$ into its even and odd part. 
\begin{theorem}\label{Kem1}[Kemer]
Every  PI algebra $R$ is PI equivalent to the {\em Grassmann envelope} $G_0\otimes A_0\oplus G_1\otimes A_1$ of a finite dimensional super--algebra $A=A_0\oplus A_1$.
\end{theorem}

 The algebra $A$ is of course not unique, nevertheless some normalisations in the choice of $A$ can be made and the purpose of this paper is to show that there is a deep geometric structure  of the algebra $F \langle  X \rangle  /I$  which reflects the structure of these normalised algebras.

A major motivation of  Kemer  was to solve the Specht problem, that is to prove 
\begin{theorem}\label{Spe}[Kemer]
All $T$--ideals are finitely generated as $T$--ideals.
\end{theorem} 
This implies that, when working with $T$--ideals, we can use the standard method of {\em Noetherian induction} that is, every non empty set of $T$--ideals contains maximal elements. 
\subsubsection{Alternating polynomials}
Kemer's theory is based on the existence of some special alternating polynomials which are not identities, that is do not belong to a given $T$--ideal $\Gamma$. So let us start by reviewing this basic formalism.\smallskip
 
Let us  fix some  positive integers  $\mu,t,s$, we want to construct  multilinear polynomials in some variables $X$ and possibly other variables $Y$.  We want to have $N=\mu t+s(t+1)$  variables $X$  decomposed as $\mu$ disjoint subsets called {\em small layers}  $X_1,\ldots,X_\mu $, with $t$ elements each and $s$  {\em big layers} $ Z_1,\ldots,Z_s$, with $t+1$ elements each,   we consider polynomials alternating in  each layer.

Such a space of polynomials is obtained, by operations  of substitution of variables,  from a finite dimensional space  $M_{\mu,t,s}(X,W)$  constructed as follows.  

Take $N+1$  variables $w_1,w_2,\ldots, w_{N+1}$ and consider the space spanned by  the $N!$  monomials
$$w_1x_{\sigma(1)}w_2x_{\sigma(2)}\ldots w_Nx_{\sigma(N)}w_{N+1},\ \sigma\in S_{N+1}. $$

In this space the subgroup  $G:=S_t^\mu\times S_{t+1}^s$   acts  permuting the monomials and thus we have a  subspace $M_{\mu,t,s}(X,W)$ of dimension $N!/ t! ^\mu(t+1)!^s$ with basis all possible polynomials alternating in these layers.

Since in general we work with algebras without 1 we also need to add all polynomials obtained from these by specialising some of the variables  $w_i$ to 1.

We obtain thus a space $M_X=M_{\mu,t,s}(X,W)$ of polynomials in $X$, and some of the $W$,  alternating in the layers of $X$, so that,  if we take any polynomial $f(X,Y)$  which is multilinear and alternating in the layers of $X$, and depends on other variables $Y$,  this polynomial is obtained as linear combination of elements of  $M_X$ after substitution of the variables $w_i$ with polynomials in the variables $Y$.

In particular this shows how  from a space $M_{\mu,t,s}(X,W)$ one may deduce, by variable substitutions,  larger spaces  $M_{\mu',t',s'}(X,W),\ \mu'\geq \mu,\ t'\geq t,\ s'\geq s.$\smallskip

Particular importance have  the {\em Capelli polynomials}, introduced by  Razmyslov, \cite{razmyslov1}
\begin{equation}\label{Capelli}C_m(x_1,x_2,\dots,x_{m};w_1,w_2,\dots,w_{m+1}):=\sum_{\sigma\in
S_{m}}\epsilon_\sigma w_1x_{\sigma(1)}w_2x_{\sigma(2)}\dots
w_{m}x_{\sigma(m)}w_{m+1}.\end{equation}
In fact this polynomial plays a role analogous to that of the classical Capelli identity, which is instead an identity of differential operators.

Since one often needs to analyse algebras without 1, it is useful to introduce the {\em Capelli list} $\mathcal C_m$  of all polynomials deduced from $C_m(x_1,x_2,\dots,x_{m};w_1,w_2,\dots,w_{m+1})$ by specialising one or more of the variables $w_i$ to 1.\smallskip

One of  the facts  of the theory, consequence of Theorem \ref{Kem},  is 
\begin{proposition}\label{cap}
A PI algebra is PI equivalent to a finite dimensional algebra if and only if it satisfies some Capelli identity.\footnote{for algebras without 1 we mean that it also satisfies all the identities of the Capelli list.}
\end{proposition}

In this paper  we want to concentrate on finite dimensional algebras so we shall from now on assume  that some Capelli identity is satisfied. In this way we do not need to  introduce super--algebras nor apply Theorem \ref{Kem1}. Nevertheless  most results  could be extended to super--algebras in a more or less straightforward way as will be presented in the forthcoming book \cite{AGPR}. 
\subsubsection{Use of Schur--Weyl duality\label{usS}} Recall that, given a vector space $V$ with $\dim V=  k$, on each tensor power  $V^{\otimes d}$  act the general linear group  $GL(V)$ and the symmetric group $S_d$  which span two algebras each the centralizer of the other.  In characteristic 0 both algebras are semi--simple and thus we have the {\em   Schur--Weyl  duality} decomposition in isotypic components, indexed by partitions  $\lambda\vdash d$  of height  $\leq \dim V=k$:   \begin{equation}\label{ScW}
V^{\otimes d}=\oplus_{\lambda\vdash d,\ ht(\lambda)\leq k} S_\lambda(V)\otimes M_\lambda.
\end{equation} The $M_\lambda$ are  the irreducible representations of the symmetric group $S_d$, constructed from the theory of Young symmetrizers, with character $\chi_\lambda$. The modules $S_\lambda(V)$, are irreducible representations of  $GL(V)$ which in fact can be thought of as {\em polynomial functors} on vector spaces called {\em the Schur functors (cf. \cite{P7})}.

We consider the free algebra $F\langle X\rangle$ as the tensor algebra $T(V)$ over an infinite dimensional vector space $V$, with basis the variables $X:=\{x_1,x_2,\ldots\}$ and take   a $T$ ideal $I$.
\begin{remark}\label{slx}
If we want to stress the basis $X$ of $V$ we also write $S_\lambda(V)=S_\lambda(X).$
\end{remark}

Since a $T$--ideal is stable under variable substitutions it is in particular stable under the action of linear group of $V$, that is   $GL(V):=\cup_m GL(m,F)$  and we can decompose   $T(V)/I$  into irreducible representations of  this group  deduced From Formula \eqref{ScW}.

If we assume that  $I$ contains a Capelli identity $C_{m+1}$  (or a Capelli list)
 we have a restriction, deduced from these Capelli identities, on the height of the partitions appearing in $T(V)/I$:  \begin{equation}\label{schi}
T(V)/I:= \oplus_d\oplus_{\lambda\vdash d,\ ht(\lambda)\leq  m} n_\lambda S_\lambda(V).
\end{equation}
One can then apply the theory of highest weight vectors, this theory belongs to the Theory of Lie and algebraic groups. In our case the notion of weight  is just the multidegree in the variables $x_i$. An element   of weight  the sequences $k_1,k_2,\ldots$ is  a  non commutative polynomial homogeneous of degree $k_i$ in each $x_i$.  

Weights are usually equipped with the {\em dominance order} which in Lie theory arises from the theory of roots, in our case simply $k_1,k_2,\ldots$ is greater than $h_1,h _2,\ldots$ in dominance order  if $k_1\geq h_1,\ k_1+k_2\geq h_1+h_2,\ldots, k_1+\ldots +k_i\geq h_1+\ldots +h_i,\ldots $.
\smallskip

Consider in $GL(V) $ the unipotent group $U$ of those (strictly upper triangular) linear transformations of type $x_i\mapsto x_i+\sum_{j<i}a_{i,j}x_i$.

One know that in the  space  $S_\lambda(V)$ the subspace $S_\lambda(V)^U$  of $U$ invariants is  1--dimensional  and generated by an element of {\em weight } $\lambda$  that is homogeneous in each $x_i$ of degree  $h_i$  where $h_i$ is the length of the $i^{th}$ row of $\lambda$, in fact this is a {\em highest weight vector} using the dominance order of weights.  Thus if  the height  of $\lambda$ is $\leq m$  this element depends only upon the first $m$ variables.  On the other hand since $S_\lambda(V)$ is an irreducible representation of $GL(V) $ it is generated by this highest weight vector. One deduces from Theorem \ref{Kem}
\begin{theorem}\label{Kem2} A $T$--ideal $\Gamma\subset F\langle X\rangle$, in countably many variables $X$ is the  ideal of identities of a finite dimensional algebra if and only if it contains a Capelli list.
\end{theorem}
\begin{remark}\label{resva}
This allows us, in order to study the multiplicities $n_\lambda$  to restrict the number of variables, hence $V$  to any chosen finite number  with the constraint to be $\geq m$. 
\end{remark}

\subsubsection{The Kemer index}  A main tool in the Theory of Kemer is given by introducing  a  pair of non negative integers $\beta(\Gamma),\gamma(\Gamma)$, called the {\em Kemer index of $\Gamma$}, invariants of a $T$--ideal $\Gamma$ which contains some Capelli identities.

These numbers give a first measure of which of the spaces $M_{\mu,t,s}(X,W)$  are not entirely contained in $\Gamma$,  (or not polynomial identities of some given algebra).
\begin{definition}\label{betaA} For every $T$--ideal $\Gamma$, which contains some Capelli identities, we let $\beta(\Gamma)$ to be the
greatest integer $t$ such that {\em for every $\mu\in\mathbb N$} there exists a
$\mu$-fold $t$-alternating (in the $\mu$ layers $X_i$ with $t$ elements) polynomial not in $\Gamma$:
$$
f(X_1,\ldots,X_{\mu},Y )\notin \Gamma.
$$
 We then let $\gamma(\Gamma)$ to be the maximum $s\in\mathbb N$ for which there exists, for all $\mu$,  a polynomial
  $
 f(X_1,\ldots,X_{\mu},Z_1,\ldots,Z_s,Y )\notin \Gamma,
 $  alternating in $\mu$ {\em small} layers $X_i$ with $\beta(\Gamma)$ elements and in $s$ {\em big} layers $Z_j$ with $\beta(\Gamma)+1$  elements.\smallskip
   
   The pair  $(\beta(\Gamma),\gamma(\Gamma))$ is the {\em Kemer index} of $\Gamma$ denoted $\mathrm{Ind}\Gamma$.\smallskip

   For an algebra $A$ we define  the Kemer index of $A$, denoted by $\mathrm{Ind}(A)$, to be  the Kemer index of the ideal of polynomial identities of $A$.
\end{definition}
We order the Kemer indices lexicographically and then observe that
\begin{remark}\label{leski}
If $\Gamma_1\subset \Gamma_2$ are two $T$--ideals we have
$$\mathrm{Ind}\Gamma_1\geq \mathrm{Ind}\Gamma_2. $$
If $\Gamma=\cap_{i=1}^k \Gamma_i$ is an intersection of $T$--ideals then
$$\mathrm{Ind}\Gamma=\max \mathrm{Ind}\Gamma_i $$

If $A=\oplus_{i=1}^k A_i$ is a direct sum of algebras 
$$\mathrm{Ind}(A)=\max \mathrm{Ind}(A_i) .$$
\end{remark}
\begin{remark}\label{leski1}
By definition, denoting $s:=\gamma(\Gamma) $, there is a minimum $\mu_0=\mu_0(\Gamma)$  such that there is no  polynomial
  $
 f(X_1,\ldots,X_{\mu},Z_1,\ldots,Z_{s+1},Y )\notin \Gamma 
 $,   alternating in $\mu\geq \mu_0$ layers $X_i$ with $\beta(\Gamma)$ elements and in $s+1$ layers $Z_j$ with $\beta(\Gamma)+1$  elements.\end{remark}
 \begin{definition}\label{KbetaA} A polynomial
  $
 f(X_1,\ldots,X_{\mu},Z_1,\ldots,Z_s,Y )\notin \Gamma 
 $  alternating in $\mu$ layers $X_i$ with $\beta(\Gamma)$ elements and in $s=\gamma(\Gamma)$ layers $Z_j$ with $\beta(\Gamma)+1$  elements with $\mu> \mu_0+1$ will be called a {\em $\mu$--Kemer polynomial}.
 
 A {\em  Kemer polynomial} is by definition a  $\mu$--Kemer polynomial   for some $\mu>\mu_0+1$.\end{definition}
 
 \begin{remark}\label{altKe}
A   Kemer polynomial $
 f(X_1,\ldots,X_{\mu},Z_1,\ldots,Z_s,Y )$, which is also linear in a variable $w$,  not in the layers $ Z_j$, has the following property. Fix one of the layers $X_i$ which does not contain the variable $w$, add to it the variable  $w$ and alternate these $t+1$ variables. If $w$ is also in no layer $X_j$, we produce a polynomial alternating in $\mu-1\geq \mu_0$ layers with $t$ elements and in $s+1$  layers  with $t+1$ elements, otherwise if $w\in X_j, \ i\neq j$ we produce a polynomial alternating in $\mu-2\geq \mu_0$ layers with $t$ elements and in $s+1$  layers  with $t+1$ elements. By the definition of $\mu_0$ and of Kemer polynomial  this is always an element of $\Gamma$, so if $\Gamma=Id(A)$ a polynomial identity of $A$.
\end{remark}\begin{example}\label{KiM}{\upshape
1)\quad If $A=M_n(F)$, then   every
$(n^2+1)$-alternating polynomial is in $\mathrm{Id}(A)$.
Conversely, using Capelli polynomials, \eqref{Capelli},  the product (taken in an ordered way)
\begin{equation}
\label{mulca0}C_{\mu,n^2}(X_1,\ldots,X_\mu,Y):=\prod_{i=1}^{\mu}C_{n^2}(x_{i,1},\ldots,x_{i,n^2}, y_{i,1},\ldots,y_{i,n^2},y_{i,n^2+1}),
\end{equation}$$\quad X_i:=\{x_{i,1},\ldots,x_{i,n^2}\} $$
evaluated in suitable $e_{ij}$'s can take any value
$ e_{hk}$ for every $\mu$ so, in particular,  it is not an identity. Hence $\beta(M_n(F))=n^2,\ \gamma(M_n(F))=0$.
}

2)\quad On the opposite if $A$ is a finite dimensional nilpotent algebra and $A^s\neq 0, A^{s+1}=0$  we see that its Kemer index is $(0,s)$.

\end{example}
In a subtle, and not completely understood way,  all other cases are a mixture of these two special cases.
\begin{remark}\label{altK} We could take a slightly different point of view and define as $\mu$--Kemer polynomials for a $T$ ideal  with Kemer index $t,s$ to be the elements of $M_{\mu,t,s}(X,W)$ which are not in $\Gamma$, then deduce some larger $T$-ideal by evaluations of these polynomials.

\end{remark}
\subsection{Fundamental algebras\label{FuAl}}
We  need to recall, without proofs,  some steps of Kemer's theory and fix some notations.

Let $A$ be a finite dimensional algebra over a field $F$ of characteristic 0, $J:=\mathrm{rad}\, A$ be its Jacobson radical  and let $\overline{A}:=A/J$, a semi--simple algebra.

By Wedderburn's principal Theorem, cf. \cite{Alb},    we can decompose    \begin{equation}\label{Made}
A=\overline{A}\oplus J=R_1\oplus \cdots\oplus R_q\oplus J,\quad R_i\quad \text{simple}
\end{equation}
where, if we assume $F$ algebraically closed, for every $i$ the simple algebra  $R_i$ is isomorphic to $M_{n_i}(F)$ for some
$n_i$.  Due to Example \ref{KiM}  2), we will assume from now on that $\bar A\neq 0$ and call it the {\em semisimple part} of $A$.
 \begin{definition}
\label{tAsA}We set
$t_A:=\dim_F  A/J=\dim_F \overline{A}$, and   $s_A+1$ be the nilpotency index
of $J$, that is $J^{s_A}\neq 0,\ J^{s_A+1}=0$.\smallskip

We call the pair $t_A,s_A$  the {\em $t,s$--index} of $A$.  
\end{definition}It is easily seen that the $t,s$ index  of $A$ is greater or equal than the Kemer index  ${\rm Ind}(A)$, so it is  important to understand when these two indices coincide.\smallskip

With the previous notations consider the quotient map  $\pi:A\to A/J=\oplus_{i=1}^qR_i$ and  let  for $1\leq j\leq q$,   \begin{equation}\label{leali}
R^{(i)}:=\oplus_{i\neq j,\ j=1}^qR_j\quad\text{and}\quad A_i:=\pi^{-1} R^{(i)}\subset A.
\end{equation}  We have for all $i$ that $A_i$ satisfies all polynomial identities of $A$ and $t_{A_i}<t_A$.\smallskip

We need to construct a further algebra  $A_0 $  which satisfies all polynomial identities of $A$ and $t_{A_0 }=t_A$ but  $s_{A_0 }<s_A$.\medskip

This algebra is constructed as  the free product $\bar A\star  F \langle  X \rangle  $  modulo the ideal generated by all polynomial identities of $A$,  for a sufficiently large $X$ (in fact we shall see that $s$ variables suffice, Lemma \ref{samid}), and finally modulo the ideal of elements of degree $\geq s$ in the variables $X$. By definition $t_{A_0 }=t_A,s_{A_0 }=s_A-1$.\smallskip

  We have by construction
 $\mathrm{Id}(A)\subset \cap_{i=1}^ 
q \mathrm{Id}(A_i)\cap \mathrm{Id}(A_0 ),$ 
and if  $\mathrm{Id}(A)=\cap_{i=1}^q \mathrm{Id}(A_i)\cap \mathrm{Id}(A_0 )  $ then  $A$ is PI equivalent to $\oplus_{i=1}^ 
q   A_i \oplus A_0   $.  All the $t,s$--indices of these algebras are strictly less than the $t,s$--index of $A$. This suggests
\begin{definition}\label{FAL1}
We say that $A$ is {\em fundamental} if  \begin{equation}\label{fundeq}
\mathrm{Id}(A)\subsetneq \cap_{i=0}^ 
q \mathrm{Id}(A_i) .
\end{equation}
A  polynomial $f\in  \cap_{i=0}^ 
q \mathrm{Id}(A_i) \setminus \mathrm{Id}(A)$ will be called {\em fundamental}.
\end{definition}
\begin{example}\label{BTM}
An algebra of block triangular matrices is fundamental.
 \end{example} By induction one has that
\begin{proposition}\label{SLM}
Every finite dimensional algebra $A$ is PI equivalent to a finite direct sum of fundamental algebras.
\end{proposition}
  
Let us now see the properties of a fundamental multilinear polynomial $f$. By definition we have an evaluation in $A$ which is different from 0.   We may assume that each variable   has been substituted to a semisimple resp. radical element, we call such an evaluation {\em restricted}.

Let us call, by abuse of notation, {\em semisimple} resp. {\em radical} the corresponding variables. \begin{lemma}\label{fulfu}  Take a non--zero restricted evaluation,  in $A$,  of a multilinear fundamental polynomial $f$.
\begin{enumerate} 
\item We then have that there is at least one semisimple variable  evaluated in each $R_i$,  for all $i=1,\ldots,q$ {\em (Property of being full)}.

\item We have exactly $s_A$ radical substitutions {\em (Property K)}. \end{enumerate}
\end{lemma}
The next result is usually presented in the literature divided in two parts, called the first and second Kemer Lemma.
\begin{theorem}\label{primoK}
A finite dimensional algebra $A$ is fundamental   if and only if its Kemer index equals the $t,s$ index.

In this case, given any  fundamental polynomial $f$, we have $\mu$--Kemer polynomials in the $T$--ideal $\langle f\rangle$ generated by $f$, for every $\mu$.

\end{theorem}
\begin{remark}\label{conv} Conversely there is a $\mu$ such that all $\mu$--Kemer polynomials are fundamental.

\end{remark}
We then introduce a definition
\begin{definition}\label{primary} A $T$--ideal $I$  is called {\em primary} if it is the ideal of identities of a fundamental algebra.

A $T$--ideal $I$ is {\em irreducible} if it is not the intersection $I=J_1\cap J_2$ of two  $T$--ideals $ J_1, J_2$  properly containing $I$.

\end{definition}
We then see 
\begin{proposition}\label{primdec}  Every  irreducible $T$--ideal containing a Capelli list is primary and every $T$--ideal containing a Capelli identity is a finite intersection of  primary $T$--ideals.

\end{proposition}
\begin{proof}
By Noetherian induction  every $T$--ideal is a finite intersection  or irreducible $T$--ideals, otherwise there is a maximum one which has not this property and we quickly have a contradiction.

If a $T$--ideal $I$ contains  a Capelli list it is the $T$--ideal  of PI's of a finite dimensional algebra, which by Proposition \ref{SLM} is PI equivalent to a direct sum of fundamental algebras.  Hence $I$ is the intersection of the ideals of polynomial identities of these algebras, since it is irreducible it must coincide with   the ideal of polynomial identities of one summand, a fundamental algebra.
\end{proof}
As in the Theory of primary decomposition  we may define an {\em irredundant}  decomposition  $I=J_1\cap \ldots\cap J_k$ of  a $T$--ideal $I$ in primary ideals.  In a similar way  every finite dimensional algebra  $B$ is PI equivalent to a direct sum  $A=\oplus_iA_i$ of fundamental algebras  which is {\em irredundant}  in the sense that $A$ is not PI equivalent to any proper algebra 
$ \oplus_{i\neq j} A_i$.  We call this an {\em irredundant sum of fundamental algebras}. \smallskip

It is NOT true that the $T$--ideal of PI of a fundamental algebra is irreducible. The following example has been suggested to me by Belov.

Consider the two fundamental  algebras $A_{1,2},\ A_{2,1}$ of block upper triangular $3\times 3$ matrices stabilizing a partial flag formed by a subspace of dimension 1 or 2 respectively. They have both a semisimple part isomorphic to $M_2(F)\oplus F$. By a Theorem of Giambruno--Zaicev \cite{GZ7}  the two $T$--ideals of PI are respectively $I_2I_1$ and $I_1I_2$  where $I_k$ denotes the ideal of identities of $k\times k$ matrices. By a Theorem of Bergman--Lewin \cite{BergLew} these two ideals are different.

Now in their direct sum $A_{1,2}\oplus A_{2,1}$ consider the algebra $L$ which on the diagonal has equal entries in the 2 two by two and in the two one by one blocks.  It is easy to see that $L$ is PI equivalent to $A_{1,2}\oplus A_{2,1}$ and that it is fundamental. Its $T$--ideal is not irreducible but it is $I_2I_1\cap I_1I_2$.

\subsubsection{The role of traces}  Let  $A$ be a fundamental algebra  over a field $L$, with index $t,s$ and  $f$ a $\mu$--Kemer  polynomial for $\mu$ sufficiently large so that $f$ is also fundamental.

It follows from Lemma \ref{fulfu}  that {\em the evaluation in the small layers factors through the radical}.  
   
 So let us fix one of the small layers of the variables, say $x_1,\ldots,x_t$, and denote for simplicity by $Y$ the remaining variables. 
 Thus,  having fixed some evaluation,  which we denote by $\bar Y$, of the variables $Y$ outside the chosen small layer   of variables,  we deduce from $f$ a multilinear  alternating map  from $A/J$  to $A$ in $t$ variables (still denoted by $f$), which  must then have  the form:
 $$ f(x_ 1,x_2, \ldots, x_t, \bar Y)=\det(\bar x_1,\ldots,\bar x_t) u(\bar Y),$$ where the $\bar x_i$ are the classes in $\bar A$ of the evaluations of the variables $x_i$ and  we have chosen a trivialization of $\bigwedge^t \bar A$.
 
  Since $f$  is fundamental  and there are $s$ radical evaluations we also have $u(\bar Y)\in J^s$.

As a consequence we deduce the important identity
\begin{equation}\label{esced}
f(zx_ 1,zx_2, \ldots, zx_t,   Y)=\det(\bar z)f(x_ 1,x_2, \ldots, x_t,  Y)
\end{equation} where $\det(\bar z)$ means the determinant of left multiplication of $\bar z$ on $A/J$.\smallskip

When we polaryze this identity we have on the right hand side the characteristic polynomial and  in particular the identity as functions on $A$:
\begin{equation}\label{TrK}
\sum_{k=1}^{t}
f(x_1,\ldots,x_{k-1},zx_k,x_{k+1},\ldots ,x_t,Y)=\mathrm{tr}(\bar z)f(x_1,\ldots,x_t,Y) ,
\end{equation} where $tr(\bar z)$ is again the trace of left multiplication by $\bar z$  in $A/J$. 

We now observe that:  the left hand side of this formula  is still alternating in all the layers in which $f$ is alternating and not an identity, since for generic $z$ we have $tr(\bar z)\neq 0$, thus it is again a Kemer polynomial.

We can then repeat the argument and  it follows that
\begin{lemma}\label{multtt} The  function
resulting by multiplying a  Kemer polynomial,  evaluated in $ A$,  by a product   $tr(\bar z_1)tr(\bar z_2)\ldots tr(\bar z_k),$ where the $  z_i $ are new variables,  is    the evaluation  in $ A$ of a  new  Kemer polynomial (involving also the variables $z_i$).
\end{lemma} 
 \smallskip

This discussion easily extends to a direct sum $A=\oplus_{i=1}^q A_i$ of fundamental algebras (with radical $J_i$)  over some field $L$ all with the same Kemer index $(t,s)$ getting a formula
\begin{equation}\label{TrK1}
\sum_{i=1}^qt_i (\bar z)f_i(x_1,\ldots,x_t,Y) =\sum_{k=1}^{t}
f(x_1,\ldots,x_{k-1},zx_k,x_{k+1},\ldots ,x_t,Y),
\end{equation}  where $f_i(x_1,\ldots,x_t,Y) $  is the projection to $A_i$ of the  evaluation $f (x_1,\ldots,x_t,Y) $, while $t_i(\bar z)$  is the trace of  left multiplication by the class of $z$ on $A_i$.  In this case if $\mu$ is sufficiently large and $f$ is a $\mu$--Kemer polynomial for one of the algebras $A_i$ then it is either a PI or a $\mu$--Kemer polynomial for each of the $A_j$. It is then convenient to think of $A$ as a module over the direct sum  of $q$ copies of $L$. Then we can write $t(\bar z):=(t_1(\bar z),\ldots, t_q(\bar z))\in L^q,$ and we have:
$$ t  (\bar z)f (x_1,\ldots,x_t,Y) =\sum_{k=1}^{t}
f(x_1,\ldots,x_{k-1},zx_k,x_{k+1},\ldots ,x_t,Y) .$$
Notice furthermore that if $\mu$ is sufficiently large and $f$ is a generic $\mu$--Kemer polynomial it is not a PI for any of the algebras $A_i$.

\subsubsection{$T$--ideals of finite Kemer index.}

We have remarked that a $T$--ideal in the free algebra in countably many variables is of finite Kemer index if and only if it contains some Capelli list.

The following fact could also be obtained from the theory of Zubrilin, \cite{Zubr}.

 \begin{proposition}\label{kemn} If $\Gamma$ is a $T$--ideal with  finite Kemer index $t,s$, there is a $\nu\in\mathbb N$  such that, every $\nu$--Kemer polynomial  $f(X_1,\ldots,X_\nu, Z_1,\ldots,Z_s,W)\in R(X)$ has the property  that, for two extra variables $y,z\notin X$  one has, modulo $\Gamma$
 $$\forall i=1,\ldots,\nu: f(zX_1,\ldots,X_\nu, Z_1,\ldots,Z_s,W)\stackrel{mod.\ \Gamma}= f( X_1,\ldots,zX_i,\ldots,X_\nu, Z_1,\ldots,Z_s,W);$$
 \begin{equation}\label{molti}
f(zyX_1,X_2,\ldots,X_\nu, Z_1,\ldots,Z_s,W)\stackrel{mod.\ \Gamma}=f(yzX_1,X_2,\ldots,X_\nu, Z_1,\ldots,Z_s,W).
\end{equation}
\end{proposition}
\begin{proof}
By Theorem \ref{Kem2}, Proposition \ref{SLM} we know that $\Gamma$ is the ideal of identities of a finite direct sum of fundamental algebras $A=\oplus_{i=1}^m A_i$.  Also  by Remark \ref{leski}  the Kemer index  of $\Gamma$ is the maximum of the Kemer indices of the fundamental algebras $A_i$.  Thus, we may assume that in the list $A_i$ the first $k$ algebras have this maximum Kemer index  and decompose $A=B\oplus C$ where $B=\oplus_{i=1}^k A_i$.  It follows that there is some $\nu\in\mathbb N$ such that any $\nu$--Kemer polynomial for $\Gamma$  is a $\nu$--Kemer polynomial for $B$  while it is a polynomial identity  for $C$.  

Then  Formula \eqref{molti} is valid if and only if it is valid  modulo $\Gamma':={\rm Id}(B)$ that is as functions on  $B$ and this is insured by Formula \eqref{esced}.
\end{proof}

\subsubsection{Generic elements} %We next apply this to $L$, a large field of rational functions over $F$  which is given by the {\em method of generic elements}.

  Let  $A  $ be a  finite dimensional  algebra  over an algebraically closed field $F$, and $\dim\!_FA=n$, fix a basis $a_1,\ldots,a_n$ of $A$.   Given $m\in\mathbb N$  (or $m=\infty$) consider  $L$,   the rational function field  $F(\Lambda)$ where $\Lambda=\{\lambda_{i,j},\ i=1,\ldots,n;\ j=1,\ldots,m\}$ are $nm$ indeterminates.\smallskip

We can construct $m$ {\em generic elements} $\xi_j:=\sum_{i=1}^n\lambda_{i,j}a_i$  for $A$, and  in $A\otimes L$  we construct the  algebra $\mathcal F_A(m)=F\langle\xi_1,\ldots,\xi_m\rangle$ generated by these generic elements. This is clearly isomorphic to the relatively free algebra   quotient of the free algebra modulo the identities of $A$ in $m$ variables.\smallskip

 If $A =\oplus_{i=1}^k A_i $ is a direct sum of finite dimensional  algebras (usually   we shall assume to be fundamental)    let  $J=\oplus J_i$ its radical and $A/J=\oplus_{i=1}^k\bar A_i=A_i/J_i$.    \smallskip

  We  may choose   a basis $a_1,\ldots,a_n$ of $A$ over $F$ union of bases of the summands $A_i$, we may also choose for each summand $A_i$ decomposed as $\bar A_i\oplus J_i$  the basis    to be formed of a basis of $J_i$  and one of $\bar A_i$.  Then when we construct $m$ generic elements   $\xi_j:=\sum_{i=1}^n\lambda_{i,j}a_i$  for $A$,   by the choice of the basis    each is the sum $\xi_j=\sum_{i=1}^k\xi_{j,i}+\eta_{j,i}$  where the $\xi_{j,i}+\eta_{j,i}$ are generic for $A_j$. The $\xi_{j,i} $ are generic for $\bar A_j$, while $\eta_{j,i}$ are generic for the radical $J_j$, and all involve disjoint variables.

For each $j=1,\ldots,k$ we also have the relatively free algebra $\mathcal F_{A_j}(m)=F\langle\xi_{1,j},\ldots,\xi_{m,j}\rangle$   and an injection 
$$ \mathcal F_A(m)\subset  \oplus_{j=1}^k\mathcal F_{A_j} (m)\subset  A\otimes_F L=\oplus_{j=1}^kA_j \otimes_FL.$$

 Notice that the radical of     $   A\otimes_F L $   is $ J \otimes_F L$ and modulo the radical this is $\bar A\otimes_FL$ which is some direct sum of matrix algebras $\oplus_i  M_{n_i}(L)$.

  The projection $p:A\to \bar A=\oplus_{i=1}^k \bar A_i $ of coordinates $p_1,p_2,\ldots,p_k$,  induces a map of algebras generated by generic elements  
$$ p:\mathcal F_A\to \oplus_i  \mathcal F_{\bar A_i}\subset \oplus_i \bar A_i\otimes_FL,\ p:\xi_i\mapsto (\xi_{i,1},\ldots,\xi_{i,k}).$$
\begin{remark}\label{quop} The Kernel of $p$ is a nilpotent ideal while the image is isomorphic to the domain  $\mathcal F_{M_n(F)}$  of generic $n\times n$ matrices where $n$ is the maximum of the degrees  $n_i$ (where $\bar A=\oplus_i  M_{n_i}(F)$).

\end{remark}

%Let $\mathcal T_i$ be the  ring of traces for the generic elements  $\xi_{j,i};\ j=1,\ldots,m$ by \ref{fgtr}  the algebra $\mathcal T_i$ is finitely generated over $F$.   

We then set \begin{equation}\label{taf}
a\in\mathcal F_A,\quad t(a):=  (t_1(a),\ldots, t_k(a))\in L^{\oplus k}
\end{equation} by setting $t_i(a)$ to be the trace of left multiplication of  the image of $a$, under $p_i$,   in the summand $\bar A_i\otimes_FL$.      We let 
\begin{equation}\label{TAA}
\mathcal T_A(m):=F[t(a)]|_{a\in\mathcal F_A}\subset L^{\oplus k}
\end{equation} be the (commutative) algebra generated over $F$ by all the elements $t(a),\ a\in\mathcal F_A(m)$. From now on we assume $m$ fixed and drop the symbol $(m)$   write simply $\mathcal F_A(m)= \mathcal F_A,\ \mathcal T_A(m)= \mathcal T_A$.
 
\begin{theorem}\label{trcealg}
$\mathcal T_A  $ is a finitely generated  $F$ algebra and  $\mathcal T_A  \mathcal F_A$ is a finitely generated module over $\mathcal T_A  $.
\end{theorem}
\begin{proof} The proof uses a basic tool of PI theory, the {\em  Shirshov basis}, that is the existence of a finite number $N$ of monomials $a_i$ in the generators $\xi_j$  such that every monomial in the variables $\xi_j$, is a linear combination with coefficients in $F$ of products of powers $a_1^{n_1}a_2^{n_2}\ldots a_N^{n_N}$ \cite{rowen1}.

We first claim that every element   $a\in \mathcal F_A$  satisfies a monic polynomial with coefficients in $\mathcal T_A  $, in fact the coefficients are polynomials in $t(a^j)$ for $j\leq  \max(\dim \bar A_i)$. 

For this let  $n_i:=\dim \bar A_i$. The projection of
$a$  in $\bar A_i\otimes_FL$  satisfies    $H_{n_i}(x) $  where     we take for $H_{n_i}(x)$   the Cayley--Hamilton polynomial induced by left multiplication on  $ \bar A_i\otimes_FL$.  This  is a universal expression in $x$ and the elements $t_i(a^j),\ j\leq n_i$ where  $t_i(a^j)$   is the trace of the   the left action on   $\bar A_i\otimes_FL$ of the projection of $a^j$. 

Thus if we use the   formal Cayley Hamilton polynomial for $n_i\times   n_i$  matrices, but using as trace of $a^j$  the $k$--tuple $t(a^j)=(t_1(a^j),\ldots, t_1(a^j))$  for all $i$ we see that, if $\bar a$ denotes the image of $a$ in $\oplus_i \bar A_i\otimes_FL$ we have $H_{n_i}(\bar a)\bar a \in \oplus_i \bar A_i\otimes_FL$  has 0 in the $i^{th}$ component, so  $\prod_{i=1}^kH_{n_i}(\bar a)\bar a  =0 $ in $p(\mathcal T_A\mathcal F_A)\subset \oplus_i \bar A_i\otimes_FL$. 

Now every element of the Kernel of $p$ is nilpotent of some fixed degree $s$  and finally we deduce that
\begin{equation}\label{CH}
(\prod_{i=1}^kH_{n_i}(a)a)^s =0,\quad\text{in}\quad \mathcal T_A\mathcal F_A. 
\end{equation} We have multiplied by $a$ since we do not assume that the algebra has a 1.

We take a Shirshov basis    for  $\mathcal F_A$
then, since we know that   $t(t(a)b)=t(a)t(b)$, it follows that $\mathcal T_A$ is generated by the traces $t(M)$  where $M$ is a monomial in the Shirshov basis with exponents less that the degree of  $(\prod_{i=1}^kH_{n_i}(x)x)^{s } $ and 
$\mathcal F_A \mathcal T_A$ is spanned over $\mathcal T_A$  by this finite number of monomials. \end{proof} 
\begin{example}\label{genma} A basic example is given by $A=M_t(F)$  the algebra of matrices. In this case the algebra of generic elements is known as the {\em generic matrices}. The  commutative algebra $\mathcal T_A(Y)$  (will be denoted by $\mathcal T_t(Y)$) equals the algebra of invariants  of  $m:=|Y|$  matrices under conjugation and the algebra $\mathcal T_A\mathcal F_A$ is the algebra of equivariant maps  (under conjugation) between $m$--tuples of matrices to matrices.

\end{example}
  
Assume now that $A= \oplus_{i=1}^mA_i$ is a direct sum of fundamental algebras. Decompose $A=B\oplus C$ where we may assume that $B= \oplus_{i=1}^kA_i$ is the sum of the $A_i$ with maximal Kemer index, the same Kemer  index as $A$  and $C$ the remaining algebras.  For $\mu$ sufficiently large a  $\mu$--Kemer polynomial for $A$   is a polynomial  identity on $C$  and either a  Kemer polynomial  or a polynomial identity for $A_i,\ i=1,\ldots,k$. So in this case we call Kemer polynomial for $A$, one   with this property.   By formula \eqref{TrK}, as extended to direct sums, a Kemer  polynomial  evaluated in $\mathcal F_A\subset A\otimes_FL=\oplus_iA_i\otimes_FL$  satisfies  formula \eqref{TrK}  where $tr(\bar z)$  is the element of  $\mathcal T_A$ of Formula \eqref{taf}. 

In fact, since such polynomials vanish on $C$, the formulas factor through $\mathcal F_B, \mathcal T_B$  which are quotients of  $\mathcal F_A,\ \mathcal T_B$.
\smallskip

We can then interpret Lemma \ref{multtt} as
\begin{corollary}\label{kapo0}
The ideal $K_A$ of  $\mathcal F_A$ spanned by evaluations of  Kemer  polynomials, is a  $\mathcal T_A$ submodule and thus a common   ideal in $\mathcal F_A$ and  $\mathcal F_A\mathcal T_A$.

In fact under the quotient map $\mathcal F_A\to \mathcal F_B$  the ideal $K_A$ maps isomorphically to the corresponding ideal $K_B$, in other words the action of $\mathcal F_A\mathcal T_A$ on $K_A$ factors through $\mathcal F_B\mathcal T_B$.
\end{corollary}
The importance of this corollary is in the fact that the non--commutative object $K_A$ is, by  Theorem \ref{trcealg},  a finitely generated module over a finitely generated commutative algebra, so we can apply to it all the methods of commutative algebra. This is the goal of the  next sections.  At this point  the ideal $K_A$ depends on $A$ and not only on the $T$--ideal but as we shall see one can also remove this dependence and define an intrinsic object which plays the same role.

\section{The canonical filtration \label{KeTe1}}
\subsection{Rationality and a canonical filtration}
We want to draw some interesting consequences from the  theory developed.\smallskip

Let   $R(X)=R :=F\langle X\rangle/I$ be a relatively free algebra  in a finite number $k$ of variables $X$.

 We have seen that $I$ is the $T$--ideal of identities in $k$ variables of a finite dimensional algebra $A=\oplus_i A_i$ direct sum of fundamental algebras, we may also choose this irredundant.
 
 Choosing such an $A$ we   write $R =\mathcal F_A(X)$ and identify $ \mathcal F_A(X)$ to the corresponding algebra of generic elements of $A$. 

We can decompose this direct sum in two parts $A=B\oplus C$  where $B$ is the direct sum of the $A_i$ with the same Kemer index as $A$  while $B$ the sum of those of strictly lower Kemer index.

We have $\mathrm{Id}(A)=\mathrm{Id}(B)\cap \mathrm{Id}(C)$ and  so  an embedding $\mathcal F_A(X)\subset \mathcal F_B(X)\oplus \mathcal F_C(X)$ of  the corresponding relatively free algebras. Let us drop $X$ for simplicity.\medskip

We can apply to $B$ Theorem  \ref{trcealg}, and embed $\mathcal F_B\subset \mathcal F_B\mathcal T_B$ which is  a finitely generated module over the finitely generated commutative algebra $\mathcal T_B$, (both graded).

Let  $K_0\subset R $  be the $T$ ideal generated by the  Kemer polynomials of $B$ for sufficiently large $\mu$. Since these polynomials  are PI of $C$ it follows that under the embedding  $\mathcal F_A(X)\subset \mathcal F_B(X)\oplus \mathcal F_C(X)$  the ideal $K_0 $ maps isomorphically to the corresponding ideal in  $\mathcal F_B$  which, by  Corollary \ref{kapo0},     is a finite module over $\mathcal T_B$.

Since $K_0\subset R $  is a $T$ ideal the algebra $R_1:= R /K_0$ is also a relatively free algebra, now with strictly lower Kemer index.

  So we can repeat the construction and let $K_1\subset R $ be the $T$ ideal with $K_0\subset K_1$ and $K_1/K_0$ the $T$ ideal in $R_1$  generated by the corresponding Kemer polynomials.

If we iterate the construction, since at each step the Kemer index strictly decreases,  we must stop after a finite number of steps.
\begin{theorem}\label{finKs}\begin{enumerate}
\item We have a   filtration $0\subset K_0\subset K_1\subset\ldots  \subset K_u=R$ of $T$ ideals,  such that $K_{i+1}/K_i$ is the $T$ ideal, in $R_i:=R/K_i$,  generated by the corresponding  Kemer polynomials for $\mu$ suitably large.

\item The Kemer index of  $R_i$ is strictly smaller than that of $R_{i-1}$.
\item Each algebra $R_i$ has a quotient $\bar R_i$  which can be embedded $\bar R_i\subset \bar R_i\mathcal T_{\bar R_i}$  in a finitely generated module over a finitely generated commutative algebra $\mathcal T_{\bar R_i}$ and such that
  $K_{i+1}/K_i$  is mapped injectively to an ideal of $ \bar R_i\mathcal T_{\bar R_i}$. 
  
 \item In particular $K_{i+1}/K_i$ has a structure of a finitely generated module over the finitely generated commutative algebra $\mathcal T_{\bar R_i}$.

\end{enumerate}

\end{theorem}
\begin{corollary}[Belov \cite{belov0}]\label{ratH} If   $R :=F\langle X\rangle/I$ is  a relatively free algebra  in a finite number of variables $X$, its Hilbert series
$$H_R(t):=\sum_{k=0}^\infty  \dim( R_i)\,  t^i$$  is a rational function of the form
\begin{equation}\label{Hser}
\frac{p(t)}{\prod_{j=1}^N (1-t^{h_j})},\ \quad  h_j\in \mathbb N,\quad p(t)\in\mathbb Z[t].
\end{equation}

\end{corollary}
\begin{proof}
Clearly $H_R(t)=\sum_{i=0}^{u-1} H_{K_{i+1}/K_i}(t)$. We know that $K_{i+1}/K_i$ is a finitely generated module over a finitely generated graded algebra $\mathcal T_{R_i}$.

If $\mathcal T_{R_i}$  is generated by some elements $a_1,\ldots, a_m$  of degrees $h_i$  by commutative algebra one has that 
$$H_{K_{i+1}/K_i}(t)=\frac{p_i(t)}{\prod_{j=1}^m (1-t^{h_j}) },\quad p_i(t)\in\mathbb Z[t].$$ Summing this finite number of rational functions we have the result.
%\begin{equation}\label{hilR}
%H_R(t)=\sum_iH_{K_{i+1}/K_i}(t)=\sum_i\frac{p_i(t)}{\prod_{j=1}^m (1-t^{h_j}) }.\end{equation}
\end{proof}
In Theorem \ref{dimRRR}
and Corollary \ref{dimRRR1} we will apply a  deeper geometric analysis in order to compute from the Hilbert series the dimension of  the relatively free algebras.\smallskip

One can generalise these results by considering a relatively free algebra in $k$ variables $X $ as multi-graded by the degrees of the variables $X $, and then we write its generating series  of the multi-grading
\begin{equation}\label{hgras}
H_R(x)=\sum_{h_1,\ldots,h_k}\dim( R_{h_1 \ldots  h_k})x_1^{h_1}\ldots x_k^{h_k}
\end{equation}this is of course the graded character of the induced action of the torus of diagonal matrices (which is a standard choice for a maximal torus  of the general linear group $GL(k)$) acting linearly on the space of variables. 

Thus  the series of Formula \eqref{hgras}  should be interpreted in terms of the representation Theory of  $GL(k)$.  In each degree $d$  the homogeneous part $R_d$  of the algebra $R$ is some quotient of the representation $V^{\otimes d},\ \dim V=k$.

By Schur--Weyl  duality discussed in \S \ref{usS}  and by Formula \eqref{schi}  we have $$R_d=\oplus_{\lambda\vdash d} m_\lambda S_\lambda(V),\quad  m_\lambda\leq  \chi_\lambda(1).$$ That is  $m_\lambda$ is $\leq $ the dimension of the corresponding irreducible representation of the symmetric group $S_d$ and in fact, by Remark \ref{resva}:
\begin{proposition}\label{icocc}
If the number of variables is larger than $m$ where $R$ satisfies the Capelli list $\mathcal C_m$  we have that $m_\lambda$ equals the multiplicity of $\chi_\lambda$ in the cocharacter of $A$ of Definition \ref{coch}.
\end{proposition} 

The character of  $S_\lambda(V)$ is the corresponding Schur function  $S_\lambda(x_1,\ldots,x_k)$ (a symmetric function) so that finally we have
\begin{equation}\label{hgrass}
H_R(x)=\sum_{d}\sum_{\lambda \vdash d}m_\lambda S_\lambda(x_1,\ldots,x_k).
\end{equation}
We have seen that the numbers $m_\lambda$ are the multiplicities of the cocharacters, so it will also be interesting to write directly a generating function for   these multiplicities. 
\begin{equation}\label{hgras1}
\bar H_R(t_1,\ldots,t_k)=\sum_{d}\sum_{\lambda \vdash d}m_\lambda t^\lambda.
\end{equation}  This has to be interpreted using the theory of highest weights. Recall that a   highest weight vector    is an element of a given representation of   $G=GL(k)$  which is invariant under $U$ and it is a weight vector under the maximal torus of diagonal matrices $(x_1,\ldots,x_k)$.

 The weight is then a {\em dominant weight}  sum of  of the fundamental weights $\omega_i:=\prod_{j=1}^i x_j$ which is the highest weight of $\bigwedge^i F^k$.  If $m_i$ is the number of columns  of length $i$ of the partition $\lambda$, then the corresponding dominant weight is $\sum_i m_i\omega_i$ that is the character $\prod_{i=1}^k(\prod_{j=1}^i x_j)^{m_i}$, we identify partitions with dominant weights and thus write $$\lambda=\sum_i m_i\omega_i,  \ t_i:=t^{\omega_i},\ t^\lambda=\prod_i t_i^{m_i}.$$ A highest weight vector $v_\lambda$ of weight $\lambda$ generates an irreducible representation $S_\lambda(F^k)$ and $  S_\lambda(F^k)^U $ is 1--dimensional  spanned by $v_\lambda$.  \smallskip

In the next paragraphs, using the Theory of highest weight vectors we shall show that also these two  functions (of Formulas \eqref{hgrass} and \eqref{hgras1}) are rational and of special type and  connected to the Theory of partition functions.

Finally in Theorem \ref{maxcK} we will apply this theory to give a precise quantitative result on the growth of the colength of $R$.
\subsection{A  close look at the filtration}
{\em Our next goal is to show that the commutative rings $\mathcal T_{R_i}$, which we have deduced from some fundamental algebras  can be derived formally only from properties of the $T$--ideals, for this we need to recall the theory of polynomial maps.}
\subsubsection{Polynomial maps}
The notion of polynomial map is quite general and we refer to Norbert Roby,  \cite{Roby} and \cite{Roby1}.

A polynomial map, homogeneous of degree $t$, $f: M\to N $  between two vector spaces factors through  the map $m\mapsto m^{\otimes t}$  to the symmetric tensors $S^t(M):=  (M^{\otimes t})^{S_t}$ and a linear map $S^t(M)\to N$. If $A$ is an algebra  both $  A^{\otimes t}  $ and $S^t(A)=(A^{\otimes t})^{S_t}$ are algebras and we have the following general fact.

\begin{definition}\label{multipo}
A polynomial map, homogeneous of degree $t$, between two algebras $A,B$,   is said to be {\em multiplicative} if $f(ab)=f(a)f(b),\ \forall a,b\in A$.
\end{definition} \begin{proposition}\label{mull}[Roby \cite{Roby1}]
Given a multiplicative   polynomial map, homogeneous of degree $t$, between two algebras $A,B$, then  the induced map $S^t(A)\to B$ is an algebra homomorphism.
\end{proposition} \smallskip

We can apply this  theory  to $A$  equal to  a free algebra $ F\langle Y\rangle $,   a positive integer $t$ and   the subalgebra of symmetric tensors  $S^t (F\langle Y\rangle) =( F\langle Y\rangle ^{\otimes t} )^{S_t}$.  We  can treat  the map $z\mapsto  z^{\otimes t}$ as a {\em universal multiplicative polynomial map}, homogeneous of degree $t$.

%
%In our case this implies that  $f$ factors through an algebra homomorphism $\bar f$,  of $\bar f: S^t (F\langle Y\rangle)\to R $.
%
%In particular, if  $R$  is commutative then $\bar f$  factors through the maximal abelian quotient of  $S^t (F\langle Y\rangle)$. 

A general theorem of  Ziplies \cite{Ziplies1}  interpreting the Second Fundamental theorem of matrix invariants of Procesi and Razmyslov, that is the Procesi-Razmyslov theory of trace identities \cite{procesi2}, \cite{razmyslov1, razmyslov2}, states the following:
\begin{theorem}[Ziplies]\label{ZiVa} The maximal abelian quotient of  $S^t (F\langle Y\rangle)$ is isomorphic to the algebra $\mathcal T_t(Y)$ of  invariants of $t\times t$ matrices in the variables $Y$.\smallskip

({\em Vaccarino})\quad The previous isomorphism is induced by the explicit multiplicative map  $$\det:F\langle Y\rangle\to \mathcal T_t(Y),\quad \det:f(y_1,\ldots,y_m)\mapsto \det f(\xi_1,\ldots,\xi_m),$$ where the $\xi_i$ are generic $t\times t$ matrices.
\end{theorem}
The second part is due to Vaccarino, \cite{Vacc}.\footnote{At the moment the Theorem  is  proved only in characteristic 0.}

 \subsubsection{Kemer polynomials} The proof of Corollary \ref{kapo0}  is based on the fact that   a  $T$ ideal $\Gamma$  can be presented as ideal of identities of a finite dimensional algebra  $A$  direct sum of fundamental algebras.
 
 We now want to show that this structure of the $T$ ideal of Kemer polynomials is independent of $A$ and thus gives some information on the possible algebras $A$ having $\Gamma$ as ideal of identities.
 
 Denote by $R(Y)$ and $R(Y\cup X)$ the relatively free algebras in the variety associated to $\Gamma$  in these corresponding variables. 
 Let us first use an auxiliary algebra $A$ and let $K_A\subset R(Y\cup X)$ be the space of Kemer polynomials previously defined starting from $A$.  Changing the algebra $A$ to some $A'$ may change the space $K_A$  but  only for the $\mu$--Kemer polynomials up to some  $\mu_1$,  we shall see soon how to free ourselves from this irrelevant constraint. 
Let $K_\nu$ be the space of $\nu$--Kemer polynomials in the relatively free algebra  $ F\langle X\rangle /\Gamma $ in which the small layers are taken from the variables in $X$ (and may depend also on the variables $Y$).

By Proposition \ref{kemn}  there is an intrinsic $\nu\in\mathbb N$  such that if $f(X_1,\ldots,X_\nu, W)\in K_\nu$
we interpret  Formula  \eqref{molti}   as follows.  If $z\in F\langle Y\rangle $ we have a linear map which we shall denote by $\tilde z$ given by (choosing one of small layers):
\begin{equation}\label{esc1}
\tilde z:  f( x_ 1, x_2, \ldots,  x_t, X,  W)\mapsto f(zx_ 1,zx_2, \ldots, zx_t, X,  W) \quad mod.\ \Gamma.
\end{equation}   The operators $\tilde z$ do not depend on the small layer chosen and commute and  the map    $z\mapsto \tilde z$,   is a multiplicative polynomial map  homogeneous of degree $t$ from the free algebra $ F\langle Y\rangle $   to a  commutative algebra of linear operators.\smallskip

 Therefore   the Theorem  of  Zieplies--Vaccarino \ref{ZiVa}, tells us that Formula   \eqref{esc1}  defines a module structure on $K_\nu$  by the algebra $\mathcal T_t(Y)$ of invariants of $t\times t$ matrices in the  variables $Y$. 
 
 This module structure does not depend on $A$ and thus is independent of the embedding of  the relatively free algebra in $A\otimes L$.
 
 On the other hand,   choose $A=\oplus_i^kA_i$  so that $\Gamma={\rm Id}(A)$, a direct sum of fundamental algebras, and $A=B\oplus C$ with $B=\oplus_i^uA_i$ the direct sum of the $A_i$ with maximal Kemer index $(t,s)$  equal to the Kemer index of $\Gamma$.
 
 If  $\nu$ is sufficiently large     we know that $K_\nu$ vanishes on $C$ and in the previous embedding  $\tilde z$ coincides  with  the multiplication by  $(\det(\bar z_i))\in L^u$. When we  polaryze from $z\in R(  Y)$ we see by Formula \eqref{TrK} that multiplication by $tr(z):=(tr(\bar z_1),\ldots,tr(\bar z_u))$  maps $K_\nu$ into $K_\nu$  and  by definition  lies in $\mathcal T_A=\mathcal T_A(Y)$ (which now depends on $A$) and in fact by definition $\mathcal T_A$ is generated by these elements. We claim that 
 \begin{lemma}\label{fait}
The action of  $\mathcal T_A$ on $K_\nu$ is faithful so that   the algebra $\mathcal T_A(Y)$ is  the quotient  of  $\mathcal T_t(Y)$ modulo the ideal annihilator of the module $K_\nu$.

This ideal is independent on $\nu$ for $\nu$ large.
\end{lemma}
\begin{proof}
Now we have constructed $\mathcal T_A$ from Formula \eqref{TAA} as contained in  the direct sum of the algebras  $ \mathcal T_{A_i},\ i=1,\ldots u.$  By definition $\mathcal T_A$ depends only from the semisimple part  of  $A$ so it is the coordinate ring of a variety  which depends only upon the  algebra $\oplus_i M_{n_i}(F)$  that is from the numbers $n_i$.\smallskip

So in the end  the action of  $\mathcal T_t(Y)$, invariants of $t\times t$ matrices  in the variables $Y$, on $K$ factors through  the map 
\begin{equation}\label{maptt}
\pi:T_t(Y)\to \mathcal T_A(Y) \subset \oplus_{i=1}^u \mathcal T_{A_i}(Y).
\end{equation}  We claim that the composition of  $\pi$  with any projection to  a summand $\mathcal T_{A_i}(Y)$ is surjective. In fact the first ring is  generated by the traces of the monomials evaluated in all $t$---dimensional representations while the second is generated by the same traces but only those representations which factor through the left action of $\bar A_i$. We shall see in \S \ref{vase}   the nature of this subvariety.\smallskip

  Let $K_{\nu,i}$ be the image of $K_\nu$ in $\mathcal F_{A_i}\mathcal T_{A_i}(Y).$  We have  that each $K_{\nu,i}$   is torsion free over $\mathcal T_{A_i}$, which is a domain, since $K_{\nu,i}\subset J_i^s\otimes L$ is contained in a vector  space over the field $L$.

  Since  $K_\nu\subset \oplus_{i=1}^uK_{\nu,i}$  and  for each $i$ the restriction to $K_{\nu,i}$ is non--zero, we finally have that
$\mathcal T_{A }$ acts faithfully on $K_\nu$ so that we have defined the homomorphism of Formula \eqref{maptt}. 

\end{proof}

We want now to free ourselves from the auxiliary variables $X$  and evaluate in all possible ways the elements of $K_\nu$ in the relatively free algebra $R(Y)$  of $A$ in the variables $Y$, obtaining thus a $T$ ideal $K_R$ in $R(Y)$.

We see that
\begin{theorem}\label{extac} The module action  of  $T_t(Y)$ on $K_\nu$ induces a unique module action on $K_R$ compatible with substitutions of variables in $Y$.
\smallskip

This action factors through a  faithful   action of its image $\bar{\mathcal T}_A(Y)$. in $\mathcal T_A(Y)$.

\end{theorem}
\begin{proof}
If we work inside the  non intrinsic algebra  generated by $R(Y)$ and $\mathcal T_{A }$ we have that $K_R$ is a $\mathcal T_{A }$ submodule and this module  structure is by definition compatible with substitutions of variables in $Y$.

On the other hand since the   elements of $K_R$ can be obtained by specializing the elements of $K$    there is a unique way in which the  module action  of  $T_t(Y)$ on $K$ can induce a module action on $K_R$ compatible with substitutions of variables in $Y$.  It is a  faithful $\bar{\mathcal T}_A(Y)$ action  since $\mathcal T_{A }$ acts faithfully  on  $K_R L$.
\end{proof}
Notice that $K_R\subset B\otimes L$ so all the computations are just for the algebra $B=\oplus_i A_i$ and from now on we shall just assume $A=B$ and $C=0$.
\subsection{Representation varieties}{\em Our next task is to describe the algebraic varieties of which the rings  $\mathcal T_{A }$ are coordinate rings.}
\subsubsection{The varieties $W_{n_1,\ldots,n_u}$\label{vase}}
For a given $t$ and $m$ consider the space $M_t(F)^m$ of $m$--tuples of $t\times t$ matrices. We think of this as the set of $t$--dimensional representations of the free algebra  $F \langle  X \rangle  $ in $m$ variables  $x_1,\ldots,x_m$.

On this space acts by simultaneous conjugation the projective linear group $PGL(t,F)$ so that its orbits are the isomorphism classes of such representations. It is well known, \cite{ArM},  that the closed orbits correspond to semi--simple representations, so by geometric invariant theory the quotient variety $V_t(m):=M_t(F)^m//PGL(t)$ parametrizes  equivalence classes of semisimple representations of dimension $t$. As soon as $m\geq 2$, its generic points correspond to irreducible representations, which give closed free orbits,  so the variety has dimension $(m-1)t^2+1$.

The coordinate ring of this variety is the ring of $PGL(t,F)$ invariants, which we shall denote by  $\mathcal T_{t,m}$ or as before $\mathcal T_{t }(Y)$, if we denote by $Y$ the $m$ matrix variables.

In characteristic 0 the algebra $\mathcal T_{t }(Y)$  is generated by the traces of the monomials in the matrix variables, while in all characteristic by work of Donkin \cite{Don}, we need all coefficients of characteristic polynomials of monomials, which can be taken to be primitive.

Given non  negative integers $h_i,n_i$ with $\sum_ih_in_i=t$  we may consider, inside the variety $M_t(F)^m//PGL(t)$,  the subvariety $W_{h_1,\ldots,h_u;n_1,\ldots,n_u}$, of semisimple representations which can be obtained as direct sum   $\oplus_ih_iN_i$, from  semisimple representations $N_i$ of dimension $n_i$ for each $i=1,\ldots,u$. Of course generically each  $N_i$  is irreducible.

It will be interesting for us the special case $h_i=n_i$ which we denote by $W_{n_1,\ldots,n_u}$.\medskip

$W_{h_1,\ldots,h_u;n_1,\ldots,n_u}$  is the natural image of the product $\prod_{i=1}^uV_{n_i}(m)$, where $V_{n_i}(m)$ is  the variety of  semisimple representations of $m$--tuples of  $n_i\times n_i$ matrices, under the map  $j:N_1,\ldots, N_u\mapsto \oplus_ih_iN_i$.

We see  in fact that this map $j$ can be considered as a {\em restriction}.

Inside  the algebra  of $t\times t$ matrices, consider the subalgebra of block diagonal matrices   $\oplus_{i=1}^u h_iM_{n_i}(F)$, where the block $M_{n_i}(F)$  appears embedded into $h_i$ equal blocks, which is isomorphic of course by definition to  $M_{n_1,\ldots,n_u}:=\oplus_{i=1}^u  M_{n_i}(F)$ we call $j_{h_1,\ldots,j_n}$ this inclusion isomorphism.  

When we restrict  the invariants $\mathcal T_t(Y)$  to this subalgebra we see that when $z$ is some polynomial in $Y$ the restriction of the function $tr(z)$ to this subalgebra equals   $\sum_{i=1}^u h_i tr_i(z)$  where $tr_i(z)$ is in the algebra $\mathcal T_{n_i}(Y)$.

This means that $\mathcal T_t(Y)$  maps to  the $G=\prod_{i=1}^u  PGL(n_i,F)$  invariants, that is 
the coordinate ring of  $\prod_{i=1}^uV_{n_i}(m)$ which is $\mathcal T_{n_1}(Y)\otimes\ldots\otimes \mathcal T_{n_u}(Y)$.

We have shown that, as  $\prod_{i=1}^uV_{n_i}(m)$ is the quotient of  $m$ copies of the space $\oplus_{i=1}^u  M_{n_i}(F)$  under the group $G=\prod_{i=1}^u  PGL(n_i,F)$ acting by conjugation, we have the commutative diagram, where the two maps $\pi_G, \pi_{PGL(t,F)} $ are quotients under the two groups   
\begin{equation}\label{cdWn}
\begin{CD}
(\oplus_{i=1}^u  M_{n_i}(F))^m@>j_{h_1,\ldots,j_n}>>M_t(F)^m\\
@V\pi_G VV @V\pi_{PGL(t,F)} VV\\
\prod_{i=1}^uV_{n_i}(m)@>j>>W_{h_1,\ldots,h_u;n_1,\ldots,n_u}
\end{CD}
\end{equation}
Every representation of the form $ \oplus_ih_iN_i$ with $N_i$ of dimension $n_i$ can be conjugated  into $\oplus_{i=1}^u h_iM_{n_i}(F)$, but usually this can be done in several different ways,  thus we have
\begin{remark}\label{samedi}
The map $j:\prod_{i=1}^uV_{n_i}(m)\to W_{n_1,\ldots,n_u}$ is surjective but it is almost never an isomorphism. It is isomorphism only when $u=1$.

We claim that the two varieties have the same dimension. For this it is enough to show that   the generic fiber is finite.

The generic fiber is obtained when   in $\prod_{i=1}^uV_{n_i}(m)$ and in $ W_{n_1,\ldots,n_u}$ we have that all the summands $N_i$ are irreducible and not just semisimple.

In this case if  we have several indices $i$ with the same $n_i$ then in the expression $\oplus_i n_i N_i$  we may permute the indices of the irreducible representations of the same dimension so they come from different ways of arranging them in the factors  $V_{n_i}(m)$.

Thus if we have certain multiplicities $h_1,\ldots,h_k$ of the different indices $n_i$ we see that the generic fiber is formed by $\prod_ih_i!$  points.

In fact even if  there are no multiplicities, so the map $j$ is birational, the same argument may show that in non generic fibers we may perform some of these permutations and so  the map is not usually bijective.
\end{remark}
\begin{example}
$t=2=1+1$,  the variety $V_1(m)$  is just affine space with coordinate ring the polynomial ring $F[x_1,\ldots,x_m]$ so $V_1(m)\times  V_1(m)$ is $2m$ dimensional affine space with coordinate ring the polynomial ring $F[x_1,\ldots,x_m,y_1,\ldots,y_m]$.

For the other ring we take     monomials in the elements $(x_i, y_i)$ which should be thought of as diagonal $2\times 2$ matrices.   Such a monomial is the pair formed by   a monomial in the $x_i$ and the same in the $y_i$. Its trace is the sum over     these two monomials, a symmetric function in the exchange $\tau$ between $x_i, y_i$.  

The map is generically 2 to 1, the image of coordinate rings is the ring of $\tau$ invariants.
\end{example}
In fact we have an even stronger statement, let $A_i$ denote the coordinate ring of  the variety $V_{n_i}(m)$  so $A_1\otimes A_2\otimes\ldots\otimes A_u$ is  the coordinate ring of  the variety $\prod_{i=1}^u V_{n_i}(m)$.
We claim that  
\begin{lemma}\label{intre}
$A_1\otimes A_2\otimes\ldots\otimes A_u$ is integral over $\mathcal T_A(m)$, spanned by monomials of degree bounded by some number independent of $m$.

\end{lemma}
\begin{proof} Consider a polynomial $$f(t)=t^m-a_1t^{m-1}+a_2t^{m-2}+(-1)^ma_m$$  with roots $x_1,\ldots,x_m$ so that the elements $a_i$ are the elementary symmetric functions in the $x_i$, we may even take the $x_i$ as indeterminates.

Given an integer $k\in\mathbb N$ and a set  $S\subset \{1,\ldots,m\}$ with $h$ elements set
$$X_S^k:=\sum_{i\in S}x_i^k$$
Consider next the polynomial of degree $N:=\binom{m}{h}$
$$t^{N}+\sum_{i=1}^N(-1)^ib_i t^{N-i}:=\prod_{S\subset \{1,\ldots,m\},\ |S|=h}(t-X_S^k).$$  The coefficients $b_i$ of this polynomial are clearly symmetric functions in the variables $x_i$  so they are expressible  as polynomials in the elements $a_i$, in fact $b_i$ is a polynomial of degree  $ki$ in the variables $x_i$ so it is a polynomial of this weight  when we give to $a_i$ the weight $i$.

Each $A_i$ is generated by the traces $tr(M_i)$ of the monomials $M$ acting on the $i^{th}$ summand.  Thus  the element $tr(M_i)$  is a sum of  $n_i$ eigenvalues out of the $d$ eigenvalues of the monomial $M$  we deduce a universal polynomial of degree $\binom{d}{n_i}$ satisfied by $tr(M_i)$   with coefficients polynomials in the elements  $tr(M^k),\ k=1,\ldots,d$.

 \end{proof}

\subsubsection{The support of Kemer polynomials}
We now apply the previous discussion to Kemer polynomials, first  let us take a fundamental algebra $D$   with  semisimple part $\oplus_{i=1}^q M_{n_i }(F),\ t=\sum_i n_i^2$.  We have constructed the map from $\mathcal T_t(Y)$ to $\mathcal T_D$.
\begin{lemma}\label{coorre}
The algebra $\mathcal T_D$ is the coordinate ring of the irreducible subvariety  $W_{n_1,\ldots,n_u}$ of  $M_t(F)^m//PGL(t)$.\end{lemma}
\begin{proof}
By definition $\mathcal T_D$ is the algebra generated by the traces of $m$--tuples of elements of   $\oplus_{i=1}^q M_{n_i }(F),\ t=\sum_i n_i^2$ acting on itself by left multiplication so the Lemma follows from the previous discussion as summarized by Formula \eqref{cdWn}. 
\end{proof} 

 Let $A=\oplus_{i=1}^u A_i$ be a direct sum of fundamental algebras all with the same Kemer index $(t,s)$.
To each $A_i$  we have associated the irreducible subvariety of $M_t(F)^m//PGL(t,F)$:
$$W_i:=W_{n_1,\ldots,n_{q_i}},\quad \bar A_i=\oplus_{j=1}^{q_i}M_{n_j}(F),\ \sum_{j=1}^{q_i} n_j^2=t.$$
We have thus the rather interesting fact. \begin{theorem}\label{intrisK}
The image of the algebra $\mathcal T_t(Y)$ acting on the space of $\nu$--Kemer polynomials, for large $\nu$, $K_R$ is the coordinate ring of  a, possibly reducible, subvariety of  $M_t(F)^m//PGL(t)$ union of the subvarieties  $W_i,\ i=1,\ldots,u$.

\end{theorem}
In other words the $\mathcal T_t(Y)$--module $K_R$ of $\nu$--Kemer polynomials, for large $\nu$   is supported on this  subvariety.

Notice that $\mathcal T_t(Y)$ is an object intrinsecally defined and then also $\bigcup_iW_i$ is intrinsic being the support of $K_R$,  the subvarieties  $W_i=W_{n_1,\ldots,n_{q_i}}$   reflect the structure of the semisimple parts  of the fundamental algebras $A_i$ which may appear as summands of maximal Kemer index  of an algebra $A$ having as identities the given $T$--ideal.

There are some subtle points in this construction, first of all by $K_R$ we mean the $T$--ideal generated by $K_\nu$ for $\nu$ large, in the sense that this variety stabilizes for $\nu$ large. It is possible that some  variety $W_i$ is contained in another $W_j$, this gives an {\em embedded component} which may not be visible just by the structure of the module $K_R$ but depends on the embedding $K_R\subset \oplus_i K_i$ which in turn depends on the particular choice of $A$  so that    the ideal  $\Gamma=\mathrm{Id}_A=\cap_{i=1}^u \mathrm{Id}_{A_i}$. This appears as some {\em primary decomposition} and it is worth of further investigation.

A specific element of $K_R$ vanishes on one of the varieties $W_i$ if and only if it is  a polynomial identity  for the corresponding summand $A_i$.\medskip

\begin{corollary}\label{intrisK1} If $R$  is the relatively free algebra associated to a fundamental algebra $A$ with semisimple part $\bar A=\oplus_{j=1}^{q }M_{n_j}(F)$, then the module of Kemer polynomial is supported on the irreducible variety 
$ W_{n_1,\ldots,n_{q }}$. 

In particular if two  fundamental algebras are PI equivalent then they have the same   semisimple part.
\end{corollary} 
In general if we have two equivalent PI algebras $A=\oplus_i A_i,\ B=\oplus_j B_j$  each a direct sum of fundamental algebras, we see that the semisimple components which give maximal varieties of representations are uniquely determined.

This answers at least partially the question  on how intrinsic are the constructions associated to a particular choice of an algebra $A$  having a chosen $T$--ideal as ideal of polynomial identities. 

  \subsubsection{Dimension}
  For an associative algebra $R$ with 1,  over a field $F$, Gel'fand--Kirillov \cite{GeKi} have defined a dimension as follows.   Let $V\subset R$ be a finite dimensional subspace with $1\in V$. Let $V_n$ denote the span of all products  of $n$ elements of $V$ and set $ d_V(n):=\dim V_n.$
 \begin{definition}\label{GKD}[Gel'fand--Kirillov--Dimension (GK--Dimension)] $$Dim\,R:=\sup_V\limsup_{n\to \infty}\log d_V(n)/\log n.$$
  
  If $R$ is generated by $V$ then   $$Dim\,R = \limsup_{n\to \infty}\log d_V(n)/\log n.$$
  \end{definition}  In general  a finitely generated non--commutative algebra may have infinite dimension or a dimension which is not an integer, \cite{bokra}.
  
 A special case is when $R$  is graded and its Hilbert series $H_R(t):=\sum_{k=0}^\infty \dim_FR_kt^k$ is a rational function of type  of Formula \eqref{Hser} (this is a rather strong constraint on $R$).
  
 Since 
$$\frac{1}{1-t^h}=\sum_{i=0}^\infty t^{ih}$$  one has 
$$\frac{\sum_{j=0}^ra_jt^j}{\prod_{j=1}^N (1-t^{h_j})}=\sum_{j=0}^ra_jt^j\prod_{j=1}^N (\sum_{i=0}^\infty t^{ih_j}).$$
The function $\prod_{j=1}^N (1-t^{h_j})^{-1}$  is the Hilbert series of the polynomial algebra in generators $x_1,\ldots,x_N$ with $x_i$ of degree $h_i$. 
Write $\prod_{j=1}^N (\sum_{i=0}^\infty t^{ih_j})=\sum_{k=0}^\infty c_k t^k$  we see that $c_k$ is a non negative integer which counts in how many ways the integer $k$ can be written in the form  $k=\sum_{j=1}^Ni_jh_j,\ i_j\in\mathbb N$.

Such a function $c_k$ is classically known as a {\em partition function}, it coincides on the positive integers with  a {\em quasi--polynomial} of degree $N-1$. Quasi--polynomial  means in this case that, when we restrict $c_k$ on each coset  of $\mathbb Z$ modulo the least common multiple of the $h_j$, on the positive integers in this coset this function coincides with a polynomial of degree $N-1$. There is an extensive literature on such functions  (cf. \cite{D6}).

If we develop the rational function of Formula \eqref{Hser} in power series $\sum_{k=0}^\infty d_k t^k$ we still have that {\em after a finite number of steps}  the  function $d_k$ coincides with   a   quasi--polynomial $D(k)$, but its degree  may be strictly lower than $N-1$.  One has % if $V$  is the span of some homogeneous generators that $V_k$ equals the subspace of elements of degree $\leq k$  equals $\sum_{i=0}^kd_i$, hence it follows as in the commutative algebra case, see \cite{Eis}, the:

%This depends on the dimension of the module $M$ of which it is the Hilbert series. This dimension has also various possible definitions. In terms of Hilbert series it is well  known (cf. Eisenbud \cite{Eis}), using the theory of systems of parameters,  that
\begin{theorem}\label{dimH} The dimension of $R$ is the order of the pole of $H_R(t)$ at $t=1$ and equals $n+1$ where $n$ is the degree of the quasi polynomial $D(k)$. 
\end{theorem}
%From this it follows that the dimension has a geometric meaning which is independent from the grading.  

If $R$  is a commutative algebra, finitely generated over a field $F$, has a finite dimension which can be defined in several ways, either as Krull dimension, that is the length of a maximal chain of prime ideals or as Gelfand--Kirillov dimension, or finally when $R$ is a domain as the trascendence degree of the field of fractions of $R$ over $F$.   For a finitely generated commutative graded algebra  $R$, the Hilbert series  is a rational function of type  of Formula \eqref{Hser}, the dimension equals the dimension of its associated affine variety $V(R)$.

For a module $M$  we have the notion of {\em support of $M$}, that is the set of point $p\in V(R)$  where $M$ is not zero, that is  $M\otimes_RR_p\neq 0$ where $R_p$ is the local ring at $p$. Then the dimension of $M$ equals the dimension of its support. The support is computed as follows\begin{remark}\label{moddi}
It is well known that a finitely generated module $M$ over a commutative Noetherian ring $R$ has a finite filtration $0\subset M_1\subset M_2\ldots \subset M_k=M$ such that $M_{i+1}/M_i=R/P_i$ where $P_i$ is a prime ideal.

If $R,M$ are graded and $R$ finitely generated, the $M_i$ can be taken graded and the Hilbert series is the sum of the Hibert series of the $R/P_i$, the dimension is the maximum of the dimensions of the $R/P_i$. 

 In the language of algebraic geometry  $R$  defines the algebraic variety $V(R)$ and the $P_i$ some subvarieties with coordinate rings $R/P_i$, one sees by induction that  $M$ is {\em supported}  on the union of these subvarieties and thus  its dimension is the dimension of its support.
\smallskip

If $M$ is torsion free  over a domain $R$ then it contains a free module $R^k\subset M$ so that $M/R^k$  is supported in a proper subvariety, hence it has lower dimension and the dimension of $M$ equals that of $R$.
\end{remark}
 We will apply this to the non commutative relatively free algebras.

\subsubsection{The dimension of relatively free algebras \label{diref}}
%THIS NOW HAS TO BE CHANGED IT ALL FOLLOWS FROM THE MORE PRECISE \ref{intrisK}.

We take as definition of dimension for a relatively free algebra, the one given by its Hilbert series, which measures growth and one can see equals the Gelfand--Kirillov dimension. It is almost never equal to the Krull dimension which instead equals the dimension of $R/J$  where $J$ is the nilpotent radical and it is well known that $R/I$ is a ring of generic matrices since the only semiprime $T$--ideals are the $T$--ideals of identities of matrices.

Let us first analyze a fundamental algebra  $D$, with semi--simple part  $\oplus_{i=1}^qM_{n_i}(F)$.

Take the relatively free algebra  in $m$ variables  $\mathcal F_D(m):=F\langle \xi_1,\ldots,\xi_m\rangle $ for $D$.  We have an inclusion of  $\mathcal F_D(m)\subset \mathcal F_D(m) \mathcal T_D(m)$ and also an inclusion of the ideal $K_D\subset \mathcal F_D(m)$ generated by Kemer polynomials.

Since  both $K_D$ and $\mathcal F_D(m) \mathcal T_D(m)$  are finitely generated torsion free modules over $  \mathcal T_D(m)$ it follows that the   dimension of the  4 Hilbert series of  $K_D,\mathcal F_D(m), \ \mathcal F_D(m) \mathcal T_D(m),\   \mathcal T_D(m)$ are all equal to the Gel'fand Kirillov  dimension of the algebras $\mathcal F_D(m), \ \mathcal F_D(m) \mathcal T_D(m),\   \mathcal T_D(m)$.  We have computed  the dimension   $ \dim \mathcal T_D(m)=(m-1)t+q  $ here $t=\sum_{i=1}^q  n_i^2$ is also the first Kemer index.
\begin{proposition}\label{dimfua}\begin{enumerate}
\item The GK  dimension of the relatively free algebra in $m$ variables for a fundamental algebra  $D$  is  $(m-1)t+q  $ where $t  $ is  the first Kemer index.

\item The GK  dimension of the relatively free algebra in $m$ variables for a direct sum $\oplus_i D_i$  of fundamental algebras is the maximum of the GK dimension of the relatively free algebras in $m$ variables for the fundamental algebras  $D_i.$

\item If $\Gamma$ is a $T$--ideal containg a Capelli list and $\Gamma=\bigcap_i\Gamma_i$  with the $T$--ideals $\Gamma_i$ irreducible then the dimension of $F\langle X\rangle/\Gamma$ is the maximum of   the dimensions of $F\langle X\rangle/\Gamma_i$, each one of these being  the relatively free algebra of a fundamental algebra.
\end{enumerate}

\end{proposition}
\begin{proof}
The first part we have just proved, as for the second  remark that the relatively free algebra in $m$ variables for a direct sum $\oplus_i D_i$ is contained in the direct sum of the relatively free algebras in $m$ variables for the summands $  D_i$, so its dimension is smaller or equal than the maximum of the dimensions of the algebras relative to the summands.  On the other hand  each  relatively free algebra  in $m$ variables for the summands $  D_i$ is also a quotient of the relatively free algebras in $m$ variables for the direct sum so the claim follows.The third part  follows from the second and the fact that an irreducible $T$--ideal containing a Capelli list is the ideal of PI's of a fundamental algebra (\ref{primdec}).
\end{proof}

We can also approach in a more intrinsic form using the filtration.
By Theorem \ref{intrisK} and the remark following  the dimension is the maximum of the dimensions of the modules $K_{i+1}/K_i$ each one is supported in a union of varieties $W_{n_1,\ldots,n_k}$ but on each of these subvarieties the module is torsion free so it has the same dimension as its support, we deduce:
\begin{theorem}\label{dimRRR}
The dimension of  $R$ is the maximum   dimension  of the varieties $W_{n_1,\ldots,n_k}$.
\end{theorem}
Now the dimension of  $W_{n_1,\ldots,n_k}$ is $mt-\sum_{i=1}^q (n_i^2-1)=(m-1)t+q$, so we see that as $m$ grows the maximum is obtained  from the factors $R_i$ for which the first Kemer index equals the Kemer index $t$ of $R$ and among these the one with maximum $q$.
\begin{definition}\label{qinv}
The integer $q$ is an invariant of the $T$--ideal, called {\em $q$ invariant}.
\end{definition}
\begin{corollary}\label{dimRRR1}
The dimension of  $R(m)$ for $m$ large is $(m-1)t+q$ where $t$ is the first Kemer index  and $q$ is the $q$ invariant.\end{corollary}
Observe that when $R$ is the ring of generic $n\times n$ matrices, then $t=n^2$ and $q=1$, the formula is valid for all $m\geq 2$. By work of Giambruno Zaicev a similar statement, that the formula is valid for all $m\geq 2$,  is true for block triangular matrices  with $q$ equal the number of blocks and $t$ the dimension of the semisimple part. 

\subsection{Cocharacters \label{icocc}}   We want now to extend  
work of  Berele \cite{berele2.2} and Belov \cite{belov0} in which they show how cocharacter multiplicities are described by partition functions.  This requires some standard preliminaries.
\subsubsection{Partition functions}
The notion of partition function can be discussed  for a sequence of integral vectors $S:=\{a_1,\ldots,a_m\},\ a_i\in\mathbb Z^p$ for which there is a linear function $f$ with $f(a_i)>0,\ \forall i$.  \footnote{This restriction is essential otherwise the value of the partition function is $\infty$} 

Then one defines   the partition function on $\mathbb Z^p$
$$b\in \mathbb Z^p,\quad P_S(b)=\#\{t_1,\ldots,t_m\in \mathbb N\,|\,\sum_{i=1}^mt_ia_i=b\}.$$
Of course $P_S(b)=0$ unless  $b\in C(S):=\{\sum_ix_ia_i\,|\, x_i\in \mathbb R^+\}$  the {\em positive cone} generated by the elements $a_i$. The assumption $f(a_i)>0,\ \forall i$ means that $C(S)$ is {\em pointed}, that is it does not contain any line (only half lines).
%\begin{remark}
%When $p=1$  and the elements $a_i\in\mathbb N$, the partition function goes back to Euler. Even in this special case this function is quite complex, on  each coset, modulo the least common multiple  of the $a_i$, it coincides (on  $\mathbb N$) with a polynomial with rational coefficients of degree $m-1$. The leading terms of these polynomials coincide and equal a formula for a volume of a variable polytope.
%
%A similar but subtler theory holds in higher dimension, \cite{D6}.
%\end{remark}

It is customary to express the partition function by a  generating function, this  is a series in $p$ variables $x_1,\ldots,x_p$.  When $b=(b_1,\ldots,b_p)\in \mathbb Z^p$ we set $x^b:=\prod_{i=1}^px_i^{b_i}$ and then setting 
  $$ \mathcal P_{S}  =\sum_{b\in \mathbb Z^p}P_S(b)x^b   $$ one sees that
   $$ \mathcal P_{S}  =\frac{1}{\prod_{i=1}^m(1-x^{a_i})} .$$
   This formal series  is also truly convergent on  some region of space. 
 We can interpret this in terms of graded algebras (or geometrically as torus embedding).   Let   $R_S=F[y_1,\ldots,y_m]$  be the polynomial algebra in $m$ variables $y_i$ to which we give a  $\mathbb Z^p$ grading by assigning to $y_i$ the degree  $a_i$.  For every graded vector space $V=\oplus_{a\in \mathbb Z^p}V_a$, for which  $\dim(V_a)<\infty,\ \forall a$, we can define its {\em graded Hilbert series}  $$H_V=\sum_{a\in \mathbb Z^p}\dim(V_a)x^a.$$  One has to be careful when manipulating such series since in general a product of two such series makes no sense, so if we have two graded vector spaces $V,W$  with the previous restriction on graded dimensions, in general $V\otimes W$ does not satisfy this restriction. The product   makes sense if we restrict to series supported in a given pointed cone, in this case we have that $H_{V\otimes W}= H_{V } H_{  W} $.\smallskip

In general let us consider a finitely generated  graded $R_S$ module $M$ then we have the
\begin{lemma}\label{grmvv} The partition function of  $S$ coincides with the Hilbert series of  $R_S$.

The Hilbert series of $M$ has the form 
\begin{equation}\label{himul0}
H_{M}  =\frac{p(x)}{\prod_{i=1}^m(1-x^{a_i})},\  p(x)\in \mathbb Z[x_1^{\pm 1},\ldots, x_p^{\pm 1} ]. 
\end{equation}  That is $p(x)$ is a finite linear combination with integer coefficients of Laurent monomials.\smallskip

If   $S\subset \mathbb N^p$ and $M$ is graded in $\mathbb N^p$ then the elements $p(x)$ are polynomials.
\end{lemma}
\begin{definition}\label{nicer}
In  \cite{berele2.2}, Berele calls a rational function of the type of Formula \eqref{himul0} with $p(x)$ a polynomial a {\em nice rational function}.

If we set all the variables $x_i$  equal to $t$  in a nice rational function $H$ we have a nice rational function of $t$, the order of the pole at $t=1$  of this nice rational function will be called the {\em dimension}  of $H$.

\end{definition}

This dimension  gives an information on the growth of  the coefficients of the generating function $H (t):=\sum_{k=0}^\infty c_kt^k=\frac{p(t)}{\prod (1-t^{h_i})}$. In fact:
\begin{proposition}\label{lafck}
The function  $c_k$  for $k$ sufficiently large is a quasi--polynomial, that is a polynomial on each coset  modulo  the least common multiple $m$ of the $h_i$. 

 If $c_k\geq 0$ for $k>>0$  then the dimension which equals the order of the pole of this function at $t=1$  equals the (maximum) degree of these polynomials plus 1. 

\end{proposition} \begin{proof}
Let $m$ be the least common multiple of the  integers $h_i$ so that
$$m=h_ik_i,\implies (1-t^m)=(1-t^{h_i})(\sum_{j=0}^{k_i-1} t^{h_ij}).$$
It then follows that   $H(t)$ can be written in the form $\frac{P(t)}{(1-t^m)^d}$ with $P(t)$ some polynomial. 
Since for $i>0$  we have $$ \frac{t^m}{(1-t^m)^i}=\frac{1-(1-t^m)}{(1-t^m)^i}=\frac{1}{(1-t^m)^i}-\frac{1}{(1-t^m)^{i-1}},$$ it then follows that $H(t)$  has an expansion in partial fractions
$$H(t)=\sum_{i=1}^d\frac{p_i(t)}{(1-t^m)^i}+q(t)$$ with all the $p_i(t)$ polynomials of degree $<m$  in $t$  and $q(t)$ a polynomial.

Then remark that the generating function of
$$ \frac{t^j}{(1-t^m)^i}=\sum_{k=0}^\infty \binom{i+k-1}{i-1}t^{mk+j}, \ 0\leq j<m$$
gives a polynomial on the coset  $\mathbb Zm+j$ with positive values of degree $i-1$.

It follows that after a finite number of steps (given by the polynomial $q(t)$) the fucntion $c_k$  is a polynomial on each coset, and of  degree $d-1$  on some cosets  where the coefficients of $p_d(t)$  are different from 0.  If  $c_k$ is definitely positive it follows that all the coefficients of $p_d(t)$ are non negative, in particular $p(1)\neq 0$  and the order of the pole at $1$ of $H(t)$ is clearly $d$. \end{proof}
%\begin{proof}
%The standard proof is by induction on $m$, we prove the second part the first is a special case. If $m=0$ then $M$ is finite dimensional, it has a basis $v_1,\ldots,v_u$  of vectors with some degrees $c_i$ and its Hilbert series is $\sum_{j=1}^u x^{c_i}$.
%
%Otherwise consider the exact sequence
%$$\begin{CD}
%0@>>> M_0@>>> M@>  y_m  >>M@>>>N_1@>>>0
%\end{CD}  $$ the middle map is multiplication by $y_m$. We have that $N_0, N_1$ are both modules over $ F[y_1,\ldots,y_{m-1}]$ for which we apply induction.
%
%One easily sees that 
%$$ 0= H_{M_0}-  H_M+x^{a_p} H_M-H_{M_1}\implies  H_M=\frac1{1-x^{a_p}}(H_{M_0} -H_{M_1})  $$ which by induction implies the claim.
%\end{proof}
%Notice that this explains also easily the formula for $H_S$.
\subsubsection{The Theory of Dahmen--Micchelli} We want to recall quickly some features of the theory of partition functions.
Let $S$ be as before a list of integral vectors.  
We  want to decompose the cone $C(S)$ into   {\em big cells}   and define its singular and regular  points.  

The big cells are defined as the connected components of the open set of $C(S)$ of {\em regular points}, which is obtained when removing from $C(S)$ the {\em singular}  points which are defined as all linear combinations of some subset of $S$ which does not span $\mathbb R^p$. 

One finally needs the notion of {\em quasi--polynomial} on $\mathbb Z^p$. This is a function $f$ on $\mathbb Z^p$ for which there exists a subgroup $\Lambda\subset  \mathbb Z^p$  of finite index, so that  $f$, restricted to each coset of  $\Lambda$ in $  \mathbb Z^p$ coincides with some polynomial. 

In the theory of partition functions plays a major role a finite dimensional space of quasi--polynomials introduced by Dahmen--Micchelli. Let us recall their Theory.

\begin{itemize} 
\item 
Given a  list of integral vectors $S:=\{a_1,\ldots,a_m\},\ a_i\in \mathbb Z^p$ spanning $\mathbb R^p$ we call a  subset  $Y$ of $S$ a {\em cocircuit} if it  is minimal with the property that $S\setminus Y$ does not span $\mathbb R^p$. The set of all cocircuits of $S$ will be denoted by $ \mathcal E(S)$.

\item  To $S$ is associated a remarkable convex polytope 
the {\em Zonotope}  $$Z(S):=\{\sum_it_ia_i\,|\,0\leq t_i\leq 1\}.$$

\item  The faces of this zonotope come in opposite pairs corresponding to the cocircuits of $S$.  Moreover $Z(S)$ can be paved by parallelepipeds associated to all the bases of  $\mathbb R^p$ which can be extracted from $S$. Each of these parallelepipeds has volume a positive integer, the absolute value of the determinant of the corresponding basis elements.
\item 
For a list of integral vectors $Y$  we then define the difference operator $$\nabla_Y:=\prod_{y\in Y}\nabla_y, \quad \nabla_yf(x)=f(x)-f(x-y).$$

\item The space $DM(S)$ of Dahmen--Micchelli is the space of integral valued functions on    $  \mathbb Z^p$ solutions of the system  of difference equations $\nabla_Yf=0,\ Y\in \mathcal E(S)$.

\item The space $DM(S)$ is a space of quasi--polynomials of degree $m-p$, it is a free abelian group of  dimension   $\delta(S)$  the weighted number of bases extracted from $S$, or the volume of  $Z(S)$.

\item In fact if we take a generic shift $a-Z(S),\ a\in \mathbb R^p$ we have that $a-Z(S)\cap \mathbb Z^p$ has exactly $\delta(S)$ elements and the restriction of  $DM(S)$ to $a-Z(S)\cap \mathbb Z^p$ is an isomorphism with the space of integral valued functions on $a-Z(S)\cap \mathbb Z^p$.

\end{itemize}

The main Theorem on partition functions for which we refer to \cite{D6} is, assuming that $S$ spans $\mathbb R^p$.
\begin{theorem}\label{parfunt}
  
 $P_S$ is supported on the intersection of the lattice $\mathbb Z^p$ with the cone $C(S)$.

 Given a big cell $\mathfrak c$ of $C(S)$ and a point $a\in \mathfrak c$ very close to 0 one has that $a-Z(S)\cap \mathbb Z^p$ intersects $C(S)$ only in 0.

 $P_S$ coincides on each big cell $\mathfrak c$ with   the quasi polynomial in the space   $DM(S)$   which is 1 at 0 and equals 0 in the other points of $a-Z(S)\cap \mathbb Z^p$.

 \end{theorem} 
 Finally the partition function may be interpreted as counting the number of integral points in the variable convex compact polytope  $V(b):=\{(t_1,\ldots,t_m)\,|\, t_i\in\mathbb R^+,\quad \sum_it_ia_i=b\}$.
 
  Thus the partition function is asymptotic to the volume  of this variable polytope. This volume function, denoted by $T_S(b)$ is a {\em spline} that is a piecewise polynomial function on $C(S)$, polynomial, in each big cell,  of degree $m-p$ and again there is a remarkable theory  behind these functions, cf. \cite{D6}.
%\subsubsection{Nice rational functions}
%We want to pass from the formula  \eqref{}  expressing the graded dimension of some relatively free algebra to a formula expressing directly the multiplicities $m_\lambda$. We do this following Berele \cite{berele2.2}.
%
%Let us say that a rational function $f(x_1,\ldots,x_p)$ in $p$ variables $x_1,\ldots,x_p$ is {\em nice} if it is symmetric and the denominator is a product of factors of type $1-x^{\underline h}$ where ${\underline h}=(h_1,\ldots,h_p),\quad x^{\underline h}=x_1^{h_1}\ldots x_p^{h_p}$.
%
%When we expand  
%$$\frac{1}{1-x^{\underline h}}=\sum_{k=0}^\infty x^{k\underline h}$$ we express this in the basis of Schur functions and have 
%$$ f(x_1,\ldots,x_p)=\sum_\lambda m_\lambda S_\lambda(x_1,\ldots,x_p).$$
%One can observe that from Formula \eqref{himul} follows
%\begin{lemma}\label{coeAl}
%$m_{h_1,\ldots,h_p}$ is the coefficient of  $x_1^{h_1+p-1}x_2^{h_2+p-2}\ldots x_p$ in the product $ f(x_1,\ldots,x_p)V(x_1,\ldots,x_p)$, where $V(x_1,\ldots,x_p)$ is the Vandermonde determinant.
%\end{lemma} It is then interesting to have a generating series directly for the $m_\lambda$, we will  discuss  this in the next paragraph. 
\subsubsection{$U$ invariants}
 In order to apply the previous theory to  cocharacters we need to recall some basic facts on highest weight vectors and  $U$ invariants, where $U$ is the subgroup of the linear group of strictly upper triangular matrices with 1 on the diagonal.
 
 Recall that we have seen, Proposition \ref{icocc}, that the multiplicity $m_\lambda$ of  a  cocharacter equals the multiplicity  of the highest weight vectors of weight $\lambda$ in the relatively free algebra in $k $ variables, if we are considering an algebra of dimension $k$ or more generally if it satisfies the Capelli list $\mathcal C_{k+1}$(Remark \ref{resva}).

  By Theorem  \eqref{finKs} the Hilbert series of $R$ is the sum of the contributions of the finitely many factors $K_{i+1}/K_i$ which are all modules  over finitely generated algebras. Moreover  these are all stable under the action of  the linear group $G=GL(k)$  so that we have some decomposition into irreducible representations
  \begin{equation}\label{cocK}
K_{i+1}/K_i=\oplus_\lambda S_\lambda(F^k)^{\oplus m_{i,\lambda}}. 
\end{equation}These multiplicities $m_{i,\lambda}$ can be computed   by taking the vector space  of $U$ invariants, $(K_{i+1}/K_i)^U=\oplus_\lambda (S_\lambda(F^k)^U)^{\oplus m_{i,\lambda}} $.

Since $\dim S_\lambda(F^k)^U=1$ is generated by a single vector of weight $\lambda$, we have  that the graded Hilbert series of $(K_{i+1}/K_i)^U$, which    is  $\mathbb N^k$ graded by the coordinates $m_i$ of the dominant weights, is the generating function of the  multiplicities  $m_{i,\lambda}$.

We thus have that $(K_{i+1}/K_i)^U$ is a  graded  module  over a graded polynomial algebra in finitely many variables $\omega_1,\ldots,\omega_k$ and the number  $m_{i,\lambda}$ is the dimension of its component  of degree $\lambda=\sum_jn_j\omega_j$.
  
 So our aim is to prove that $(K_{i+1}/K_i)^U$ is   finitely generated as  graded  module  over the polynomial algebra in the variables $\omega_1,\ldots,\omega_k$.
  
 Then we can apply  lemma \ref{grmvv}  and Theorem \ref{parfunt} which properly interpreted give the desired result.\smallskip
 
 This actually is  a standard fact on reductive groups and let us explain its proof.

%Consider the coordinate ring $F(GL(k))$  of the linear group, concretely it is the polynomial ring over $F$  in $n^2$ variables $x_{i,j}$ with the determinant of the generic matrix $(x_{i,j})$ inverted. 
%
%On this algebra the group $GL(k)$  acts both on the right and on the left and  the polynomial ring $F[x_{i,j}]$ has the Cauchy decomposition $F[x_{i,j}]=\oplus_\lambda S_\lambda(F^k)\otimes S_\lambda((F^k)^*)$.
%
Let $G$  be a reductive group in characteristic 0.  Its coordinate algebra $F(G)$ decomposes, under left and right action by $G$,  as $F(G)=\oplus V_i\otimes V_i^*$, where $V_i$ runs over all irreducible rational representations of $G$.

 If we fix a Borel subgroup $B$ with unipotent radical $U$   the algebra $F(G )^U=\oplus V_i\otimes (V_i^*)^U$  ($U$ acting on the right)  is a  finitely generated algebra over which $G$ acts on the left. In fact it is generated by the irreducible representations associated to the fundamental weights.
    
    This algebra is the coordinate ring of an affine variety $\overline{G/U}$ which contains as open orbit $G/U$.  
    
    For every subgroup $H$ of $G$  consider the invariants  under the right action $F(G)^H$. Given a representation $V$ of $G$, the group  $G$ then acts diagonally on   $V\otimes F(G)^H$ and:
    
%    In our case  $G=GL(k,F)$ has dimension $k^2$ while $U$ is the group of striclty triangular matrices of dimension $\frac{k^2-k}2$  so  $F(G )^H$ has dimension $\frac{k^2+k}2.$   Moreover 
\begin{lemma}\label{grco}
For every representation $V$ of $G$ we  have   the equality
\begin{equation}\label{UinvG}
V^H=(V\otimes F(G)^H)^G. 
\end{equation}

In fact this is given by restricting to $(V\otimes F(G)^H)^G$ the explicit map  $$\pi:  V\otimes F(G )\to V, \quad \pi:v\otimes f(g)\mapsto vf(1).$$

If $V$ is an algebra with $G$ action as automorphisms this identification is an isomorphism of algebras.
\end{lemma}
\begin{proof}
The space of polynomial maps from $G$ to $V$  is clearly $ V\otimes F(G)$ where $v\otimes \phi$  corresponds to the map $g\mapsto \phi(g)v$.  We act on such maps with $G\times G$, the right action  is  $f^h(g):=f(gh)$ while for the left action we use  $^h\!f(g):=hf(h^{-1}g).$ These two actions commute, and the left action on maps is the tensor product of the action on $V$ and the left action on $F(G)$, that is if $f=v\otimes\phi$ we have $^hf=hv\otimes ^h\phi$. .

A map $f$ is $G$ equivariant, where on $G$ we use the left action if $f(hg)=h f(g)$, this means that $f$ is invariant under the left action on maps.

For such a map we have $f(g)=f(g1)=g f(1)$, conversely given any vector $v\in V$ the map $g\mapsto gv$ is $G$ equivariant so  we have a canonical identification $$j:(V\otimes F(G))^G\cong V, \ j(f):=f(1), \iff \ j:v\otimes \phi \mapsto  \phi(1)v.$$ 

 Moreover the map $j$ is $G$ equivariant if on $(V\otimes F(G))^G$ we use the right $G$ action, since for an equivariant map $f$ we have $f^h(g):=f(gh)$ (induced from right action), which maps to $f^h(1)=f(h)=h f(1).$

Under this identification, if  $a=\sum_jv_j\otimes \phi_j(g)\in V\otimes F(G)^H$  is invariant under right action by some subgroup $H$ it means that for the corresponding map $f:G\to V$ we have that $f(gh)=f(g),\ \forall h\in H$, thus $f(h)=f(1)$  and, if $f$ is $G$ equivariant we see that this is equivalent to the fact that    $\sum_jv_j\otimes \phi_j(1)\in V^H$.

As for the second statement it is enough to observe that $\pi$ is a homomorphism of algebras.
\end{proof}
This allows us to replace for $G$ modules, the invariant theory for $U$, which is not ruled by Hilbert's theory since $U$ is not reductive, with the one of the reductive group $G$, and obtain in this way the desired statements on finite generation. In fact a simple argument as in Hilbert theory shows the following.

Let $M$ be a module finitely generated over a  finitely generated commutative algebra $A$. Assume that a linearly reductive group $G$ acts on $M$, and on $A$ by automorphisms in a compatible way, that is $g(am)=g(a)g(m),\ a\in A,\ m\in M$  then
\begin{lemma}\label{fgenmo}
 The space of invariants $M^G$  is  finitely generated as a module over the finitely generated  algebra $A^G$.
\end{lemma}
\begin{proof}
Consider the $A$ submodule $AM^G$ of $M$. It is finitely generated by some elements $m_1,\ldots,m_k\in M^G$.

Thus if $u\in M^G$ we have $u=\sum_i a_i m_i,\ a_i\in A$. Now the map   $A^k\to M$  given by $\sum_i a_i m_i$ is $G$ equivariant so it commutes with the projection to the invariants  (in $A$ called the Reynolds operator $R$), so $u=\sum_i R(a_i) m_i,\ R(a_i)\in A^G$ hence $M^G$ is generated over $A^G$ by the elements $m_i$.
\end{proof}

At this point we only have to apply the theory of graded modules over graded algebras, here the grading is by the semigroup of dominant weights (which one can identify to $\mathbb N^k$) and use the  fact that the algebra  $F(G)^U$  is finitely generated by elements  which have as weight the fundamental weights. Thus if $V$ is a finitely generated module over the finitely generated algebra $A$ with $G$ action, we have that $V\otimes F(G)^U$  is a finitely generated module over the finitely generated algebra $A\otimes F(G)^U$ with diagonal $G$ action.
We deduce
\begin{theorem}\label{fingU}  $V^U$ is a finitely generated module over the finitely generated algebra $A^U$.

\end{theorem}
\begin{corollary}\label{raUU}  If $V$ and $A$ are as before, and $V=\oplus_\lambda m_\lambda S_\lambda(F^k)$, the generating function $\sum_\lambda m_\lambda t^\lambda$ is a rational function, in the variables $t_i:=t^{\omega_i}$ with denominator a product of $1-t^{\mu_i}=1-t^{\sum_im_i\omega_i}=1-\prod_it_i^{ m_i}$.

The    $\mu_i=\sum_im_i\omega_i,\ m_i\in\mathbb N$ are the dominant weights of some finite set of irreducible representations generating $A$ as algebra.

\end{corollary}
\begin{proof} If  $V=\oplus_\lambda m_\lambda S_\lambda(F^k)$ we have $V^U=\oplus_\lambda m_\lambda S_\lambda(F^k)^U$, and  $S_\lambda(F^k)^U$ is 1--dimensional generated by a vector of weight $\lambda$ so  $\sum_\lambda m_\lambda t^\lambda$ is the Hilbert series of the graded module $V^U$  (graded by dominant weights). Then compute the generating function of $V^U$  using its identification with $(A\otimes F(G)^U)^G$ and then apply Lemma \ref{grmvv} and Lemma \ref{fgenmo}, we only  need to remark that the torus $T$ acts on $(A\otimes F(G)^U)^G$  by acting  on  $F(G)$ on the right and   under the identification the weight is preserved.

\end{proof}

In the algebras which we are studying and the various graded objects associated we have  the action of $Gl(k)$ which is in fact a polynomial action, that is no inverse of the determinant appears.

The action of the torus  of diagonal matrices $a_1,\ldots,a_k$  determines the weight decomposition so the multi--grading and the grading.

The fundamental weight $t_i:=\omega_i=a_1a_2\ldots a_i$  has degree $i$  so the ordinary Hilbert series of  $V^U$, $\sum_n\dim V^U_n$  is given by substituting $t_i:=t^{\omega_i}\mapsto t^i$ so that $t^{\sum_i m_i\omega_i}\mapsto t^{\sum_i m_i i} $. 

Notice that in degree  $n$  the dimension  of $V^U_n$ equals the length or number of  irreducible components in which $V_n$ decomposes.
\begin{corollary}\label{lengthes}
The length of  $V_n$ is a nice rational function with numerator  a polynomial in $\mathbb Z[t]$  and denominator a product as   factors of type $1-t^{h_i}$. 

%Same statement for the colength of a PI algebra satisfying some Capelli identity.
\end{corollary}

\subsubsection{Cocharacters}
We want to apply the previous Theory to cocharacters.  Let $A$ be a PI algebra satisfying a Capelli identity (or rather a Capelli list)  $C_m$. By Kemer's theorem this is in fact PI equivalent to a  finite dimensional algebra.

We have then that the cocharacter $\chi_k=\sum_{\lambda\vdash k\,|\, ht(\lambda)<m}n_{k,\lambda}\chi_\lambda$ is a sum of irreducible characters   $\chi_\lambda$  associated to partitions with  height $<m$.

Consider the generating function $\sum_k(\sum_{\lambda\vdash k\,|\, ht(\lambda)<m}n_{\lambda}t^\lambda)$. We write   $t^\lambda=\prod_{i=1}^{m-1}t_i^{n_i}:=\underline t^{\underline n},\ \underline n:=(n_1,\ldots,n_{m-1})$  where $n_i$ equals the number of columns of $\lambda$ of length $i$.   

 \begin{theorem}\label{grmcoc} The generating function $\sum_{\lambda}n_{\lambda}t^\lambda$ of the multiplicities of the cocharacters is a nice rational function (\ref{nicer}) that is it has the form 
\begin{equation}\label{himul}
H_{co}  =\frac{p(t)}{\prod_{i=1}^m(1-\underline t^{\underline n_i})},\  p(t)\in \mathbb Z[t_1,\ldots, t_{ m- 1} ]. 
\end{equation} 

The generating  function of the colengths is also a nice rational function (\cite{berele2.2}).  \end{theorem}
\begin{proof}
From what we have seen $n_\lambda$ is the multiplicity  of the irreducible representation $S_\lambda(X)$  in the relatively free algebra $\mathcal F_A(X)$ associated to $A$.  We know that this multiplicity equals the multiplicity of $S_\lambda(X)^U$ (Remark \ref{slx} and \ref{resva}) which equals $S_\lambda(\{x_1,\ldots,x_{m-1}\})^U$.  Now we have for the free algebra in $m-1$   variables the canonical filtration, where each $K_i/K_{i-1}$ is a $GL(k)$  module, so it has a generating series of cocharacters $H_{co} ^i$ and $H_{co} =\sum_iH_{co} ^i$. Furthermore $K_i/K_{i-1}$  satisfies the hypotheses of Theorem \ref{fingU} with $V=K_i/K_{i-1}$ and $A=\mathcal T_i$.
So for each $K_i/K_{i-1}$ we may apply corollary \ref{raUU}, giving a contribution  to Formula \eqref{himul} of the same type. \smallskip

For the colength we apply Corollary \ref{lengthes}.
\end{proof}

\begin{remark}\label{codm}  In principle the rational function describing the generating series of the cocharacters contains all the information  on the codimension.  The series of codimensions is obtained from the series of cocharacters by a formal linear substitutions  of a monomial $t^a$  by  $\chi_a(1)t^{|a|}$.    It may be worth of further investigation  the properties of this  linear map on the space of power series which are expressed by rational functions.

\end{remark}
 \subsection{Invariants of several copies of $V$}  In the next section we shall deduce some precise estimates on the dimension (cf. Definition \ref{nicer}) of the  rational functions expressing cocharaters and colength for a fundamental algebra. We need some general facts first.
 Let us ask the following question, let $V$  be a vector space of some dimension $k$ and $G$ a semisimple group  acting on $V$. The invariants of $m$ copies  $V^m$  under the action of $G$ are also a representation of  $GL(m,F)$,  in fact from Cauchy's Formula we have
 $$S[(V^*)^m]=\oplus_\lambda S_\lambda(V^*)\otimes S_\lambda(F^m) \implies  S[(V^*)^m]^G=\oplus_\lambda S_\lambda(V^*)^G\otimes S_\lambda(F^m) .$$  If we are interested in understanding  the multiplicity with which a given representation  $S_\lambda(F^m)$ appears we know that it equals
  $\dim  S_\lambda(V)^G$ provided that $m$ is larger than the height of $\lambda$  but it may appear only if the height of  $\lambda$ is $\leq \dim V$.  Thus this multiplicity stabilizes  for $m\geq \dim V$. On the other hand if $U$  is the unipotent group of   $SL(k,F)$  of strictly upper triangular matrices we have $$S[(V^*)^k]^{G\times U}=\oplus_\lambda S_\lambda(V^*)^G\otimes S_\lambda(F^k)^U,\ \dim S_\lambda(F^k)^U=1. $$Thus the generating function of these multiplicities is  the  generating function of  $S[(V^*)^m]^{G\times U}$.
  
  In particular for the growth we need to compute the dimension of the algebra   $S[(V^*)^m]^{G\times U}$. This we compute as follows, from the previous section we have that 
  \begin{equation}\label{uinv}
S[(V^*)^k]^{G\times U}=(S[(V^*)^k]\otimes F(SL(k))^U) ^{G\times SL(k) }. 
\end{equation}
  
\begin{theorem}\label{genfr} If $G$ acts faithfully on $V$ and 
$G\times {SL(k)}$ acts freely on a non empty  open set  of  the variety $V^k\times SL(k)/U$ the dimension of $S[(V^*)^k]^{G\times U}$ is    \begin{equation}\label{dqv}
 \frac{k^2+k }2 -\dim(G) .
\end{equation} 
\end{theorem}\begin{proof} From Formula \eqref{uinv} this dimension is the dimension of the quotient variety of $G\times SL(k)$ acting on the variety $V^k\times W$, where $W$ is the variety  of coordinate ring  $F(SL(k))^U$  which contains  $SL(k)/U$ as dense open set so its dimension is $ (k^2-1)-\frac{k^2-k}2= \frac{k^2+k-2}2  $.

 The group $SL(k,F)$ is semisimple and simply connected, so by a Theorem of Popov (cf. \cite{Popov} Corollary of Proposition 1) the coordinate ring of  $SL(k,F)$ is factorial, then, since $U$ is a connected group, also the ring $F(SL(k))^U$  is factorial hence the variety $V^k\times W$ is factorial.\smallskip

%Therefore  the algebra $S[(V^*)^k]^{G\times U}$  is the coordinate ring of the quotient  under the group $G\times {SL(k)}$ of the factorial variety $V^k\times W,\ W:= \overline{SL(k)/U}.$ Notice that $\dim W= (k^2-1)- \frac{k^2-k}2= \frac{k^2+k-2}2$.
%

For a semisimple group $H$ acting on an irreducible affine variety $X$ which is also factorial and with the generic orbit equal to $H$ or just with finite stabilizer one knows, by another Theorem of Popov, cf.  \cite{Popov0},  that the generic orbit is closed  so equals the generic fiber and, by Hilbert's theory, the quotient variety has dimension $\dim W-\dim H $.

These hypotheses are satisfied and we have in our case $\dim W= k^2+\frac{k^2+k-2}2$ and since $H=G\times SL(k)$ we have $\dim H=\dim G+k^2-1$. The formula follows.\end{proof}

  From this Formula we see that only for somewhat small $G$ we may have this strong condition.   It is then useful the following criterion.
\begin{proposition}\label{gensta} If $G$ acts faithfully on a space $V$ of dimension $k$ and  the generic stabilizer   of the action of  $G$ on $V$ is a torus, $G\times {SL(k)}$ acts freely on a non empty  open set  of $V^k\times SL(k)/U$ and $G\times U$ acts freely on a non empty  open set  of $V^k$. 
\end{proposition}
\begin{proof}
%As for the generic orbit of  $G\times {SL(k)}$ on $V^k\times W$ we may look at  the generic orbit on the open set  $V^k\times  SL(k)/U .$
 In fact let us look at  the stabilizer of   $((v_1,\ldots, v_k),U)$  it is the subgroup of  $G\times U$ stabilizing $(v_1,\ldots, v_k)$. Then, for a generic choice of $(v_1,\ldots, v_k)$ this  is the stabilizer of  a generic orbit of $G\times U$  on $V^k$ so it is enough to show that  this is trivial.   

The action  of $(g,u)$ on a vector  $(v_1,\ldots, v_k)$  is the action of $u$ on $(gv_1,\ldots, gv_k)$. Moreover $u=1+\Lambda$ is some triangular matrix      with $\lambda_{j,i}\neq 0, \implies j<i$ so finally.
\begin{equation}\label{lacc}
(g,u) (v_1,\ldots, v_k)= (w_1,\ldots, w_k),\quad w_i=gv_i+\sum_{j<i}\lambda_{j,i}gv_j.
\end{equation}
Hence if $(g,u)$   stabilizes the vector $(v_1,\ldots, v_k)$ we must have $v_i=w_i\,\ \forall i$. 

            In particular $v_1=gv_1, v_2=gv_2+\lambda v_1$. So since $v_1$ is generic $g$ is a semisimple element being in a torus. Decompose $V=V^g\oplus V_g$  the invariants and a stable complement, write each $v_i=a_i+b_i,\ a_i\in V^g,\ b_i\in  V_g$.

Consider thus $a_2+b_2=v_2= gv_2+\lambda v_1=a_2+gb_2+\lambda a_1$, this implies 
 that  $b_2-gb_2=\lambda a_1\in V_g\cap V^g=\{0\}$ implies $b_2-gb_2=0$ but this implies $b_2\in V^g$  hence $b_2=0$.  Since $b_2$ is generic in $V_g$ this  implies $V_g=0$  or  $g=1$. Thus form Formula \eqref{lacc} we deduce $\sum_{j<i}\lambda_{j,i} v_j=0,\ \forall i.$
 
 Then since  $(v_1,\ldots, v_k)$ are generic they are linearly independent and this  implies $\lambda_{j,i}=0$ and also $u=1$.
\end{proof}
 
\subsubsection{Colength}
We want to investigate now the colength of $R=\oplus_{n=0}^\infty R_n$ where $R$ is the relatively free algebra, in $k$ variables, of some  finite dimensional algebra. By definition the colength is the function $\ell(R_n)$ of $n$ which measures the number of irreducible representations of  $GL(k,F)$  decomposing the part $R_n$ of degree $n$. If $k$ is larger than the  degree $m$  of a Capelli list satisfied by $A$ we also know that the colength stabilizes.

Let $S_A =\otimes_{i=1}^q\mathcal T_{n_i}(m)$ be  the ring of invariants   of $m$ copies of  a semisimple algebra $ A =\oplus_{i=1}^qM_{n_i}(F)$  under its automorphism group $G=\prod_i PSL(n_i)$. %   so we may apply  Corollary \ref{raUU} and we have that the generating series of the colength is a rational function 
%This implies that the colength stabilizes to a quasi--polynomial whose degree equals the dimension of  the algebra $T_{R_i}^U.$  
   \smallskip

 The algebra $S_A^U $ comes from Formula \eqref{uinv}  for $V= A =\oplus_{i=1}^qM_{n_i}(F)$ under the group $G=\prod_i PSL(n_i)$ which acts faithfully on $\bar A$ and it is semisimple.  Moreover the generic element  of $ A$  is a  list of matrices each with distinct eigenvalues so that the stabilizer is a product of maximal tori. We thus  have verified all the properties of  Proposition \ref{gensta}  we can thus apply Theorem  \ref{genfr}   and then the formula for the dimension of   $S_A^U $ is given by Formula  \eqref{dqv}   where $k=\dim   A=t=\sum_in_i^2$  while $\dim G= \sum_{i=1}^q(n_i^2-1)=t-q.$

We finally get for the quotient the dimension
\begin{proposition}\label{dimRa} If $A=\oplus_{i=1}^qM_{n_i}(F),\ G=\prod_i PSL(n_i)$   the dimension of the algebra $S_A^U $ of Formula \eqref{uinv} equals 
$  \frac{t^2-t}2+q. $
\end{proposition}\begin{proof}
Here we apply Formula \eqref{dqv} with  $k=t$ and $\dim G= t-q$.
\end{proof}
We  now want to apply this to $\mathcal F_A$   the relatively free algebra of a fundamental algebra  $A$ and the  trace ring $\mathcal T_A$.

\begin{proposition}\label{colfa} If $\mathcal F_A$ is the relatively free algebra of a fundamental algebra  $A$ with $\bar A=\oplus_{i=1}^q M_{n_i}(F)$ the dimension of the rational function of colengths is
$$\frac{t^2-t}2+q,\ t=\sum_{i=1}^q  n_i^2. $$

\end{proposition}\begin{proof} By lemma \ref{intre}     the ring of invariants  $S_{\bar A} $ is a finite (torsion free)  module over the trace ring $\mathcal T_{A}$ so, by Theorem \ref{fingU},  $S_{\bar A} ^U $  a finite (torsion free) module over $\mathcal T_{A}^U$ and hence the dimension of $\mathcal T_{A}^U $ equals that of $S_A^U.$
The proposed dimension is the dimension of the colength function of  $S_A$ hence also of $\mathcal T_A \mathcal F_A$, which is a finite torsion free module over $\mathcal T_A$. But also $\mathcal F_A$ contains an ideal $K_A$ which is an ideal for $\mathcal T_A \mathcal F_A$ hence it has the same dimension.

\end{proof} 
For a general relatively free algebra $R$,
 using the standard filtration we have that  the colength of $R$ is the sum of the colengths of  the factors $K_{i+1}/K_i$ and we have seen that  the generating function $H_{\ell(R)}:=\sum_{n=0}^\infty \ell(R_n)t^n$ is a nice   rational function with denominator a product of factors $1-t^{h_i}$.
So, for large $n$  the colength $\ell(R_n)$ is a quasi--polynomial of some degree $d-1$ where $d$ is the order of the pole of $H_{\ell(R)}$ at $t=1$.

Each one of these factors $K_{i+1}/K_i$ is a finite  module over some finitely generated algebra $T_{\bar R_i}\subset \oplus_j \mathcal T_{A_j}$, where $\bar R_i$ quotient of $R_i$ is the relatively free algebra associated to an algebra $A=\oplus_jA_j$, direct sum of fundamental algebras $A_j$ all with the same Kemer index, the one of $R_i$.   $T_{\bar R_i}\subset \oplus_j \mathcal T_{A_j}$ is the coordinate ring of some union of the varieties $W_j$ associated to the fundamental algebras $A_j$.  By Theorem \ref{finKs} Moreover $K_{i+1}/K_i\subset \oplus_j(K_{i+1}/K_i)_j$, and  each $(K_{i+1}/K_i)_j$ is torsion free over the corresponding  $\mathcal T_{A_j}$.  We claim that the  number $d$  is also the maximum of the order of the pole of the Hilbert series of the colength  of $T_{A_j}$ hence the dimension of $T_{A_j}^U$.  
Let us summarize these results for $R(m)$ a relatively free in $m$ variables satisfying some Capelli list $\mathcal C_{k+1}$.
 Denote by $\bar R_i$  the quotients of the standard filtration

 \begin{theorem}\label{maxcK} The generating function of the colength of $R(m)$ is a nice rational function which stabilizes for $m\geq k$.\smallskip

 When $m\geq k$   this  rational   function has dimension  $\max(\frac{t_i^2-t_i}2+q_i) $  where $t_i$ is the first Kemer index of $\bar R_i$  while $q_i$  is its    $q$ invariant (Definition \ref{qinv}).
\end{theorem}
\begin{proof}
The only thing which requires some proof is the dimension.   This computation depends on the following fact  the dimension of a nice rational   function  associated to a generating sequence $\sum_i c_it^i, c_i>0$  is the order of the pole at $t=1$. Hence one has easily, from Proposition \ref{lafck}, that the dimension of the sum of several   nice rational   functions associated to   generating sequences $\sum_i c_it^i, c_i>0$  equals the maximum of these dimensions. In our case one has  to compute the maximum arising  from the colength in the standard filtration and finally the argument is, using the previous discussion for fundamental algebras, like the argument of Theorem \ref{dimRRR}.
\end{proof}
\section{Model algebras\label{modA}}\subsection{The canonical model of fundamental algebras\label{model}}
We want to discuss now the problem of choosing special fundamental algebras in the PI equivalence classes. In corollary \ref{intrisK1} we have seen that two PI--equivalent fundamental algebras have isomorphic semi--simple parts. Thus  it is natural  to study  fundamental algebras $A$ with a given fixed semi--simple part $\bar A=\oplus_{i=1}^qM_{n_i}(F)$ so that $\beta(A)=t= \sum_{i=1}^q n_i^2$.
\smallskip

From Lemma \ref{fulfu}  it follows that the  second Kemer index $\gamma(A)$, which for a fundamental algebra coincides with the maximum $s$ for which $J^s\neq 0$, must be $\geq q-1$.  The case $\gamma(A)=q-1$ is attained by upper triangular matrices  with semi--simple part $\bar A$.\smallskip

So our present goal is to analyse fundamental algebras with given $\bar A$ and Kemer index $t=\dim(\bar A),s\geq q-1$. We use now Definition \ref{primary} and Proposition \ref{primdec}.

Consider a  $T$--ideal $I$   of identities of a fundamental algebra $A=\bar A\oplus J$, we have  seen in  Corollary \ref{intrisK1} that  the semisimple part $\bar A$  is determined by $I$.

We want to construct a canonical fundamental algebra  having semisimple part $\bar A$ and $I$  as  $T$--ideal $I$   of identities. This algebra is constructed as in \S \ref{FuAl}. \smallskip 

First we construct a universal object.  Take  the free product $\bar A\star  F \langle  X \rangle  $,  for  $X=\{x_1,\ldots,x_m\}$.  We assume $m\geq q$ where $q$ is the number of simple blocks of $\bar A$.

We now take this free product modulo the ideal of elements of degree $\geq s+1$ in the variables $X$, where $s$ is some fixed integer, call $\mathcal F_{\bar A,s}(X)$ the resulting algebra.  

$\mathcal F_{\bar A,s}(X)$  is a finite dimensional algebra with semisimple part $\bar A$ and Jacobson radical  $J$  of nilpotency $s+1$  generated by the $x_i$. It satisfies a universal property among such algebras.

\begin{remark}\label{uniFx}
Given a finite dimensional algebra $A$, with semisimple part $\bar A$ and Jacobson radical  $J$  of nilpotency $s+1$ any map $X\to J$,   extends to a unique  homomorphism of  $\mathcal F_{\bar A,s}(X)$ to $A$  which is the identity on $\bar A$.
\end{remark}

Given a finite dimensional algebra $A$ with given $t,s$ index we define $\mathcal A_s(X)$ to be  $\mathcal F_{\bar A,s}(X)$ modulo the ideal generated by all polynomial identities of $A$. 

The Jacobson radical $J_s$ of $\mathcal A_s(X)$ (and also of $\mathcal F_{\bar A,s}(X)$)   is the ideal generated by the $x_i$  and its semisimple part is $\bar A$.

 By construction and Remark \ref{uniFx},   given any list of elements $a_1,\ldots,a_m\in J$ there is a morphism  $\pi:\mathcal A_s(X)\to A$ which is the identity on $\bar A$  and maps $x_i\mapsto a_i$.

For $m$  sufficiently large  this morphism may be chosen so that the radical $J_s$  maps surjectively to the radical $J$ so also the map of $\mathcal A_s(X)$ to $A$ is  surjective, hence $\mathcal A_s(X)$ satisfies the same PI as $ A$.  

\begin{definition}\label{fundal} 
We now define  $\mathcal F_{\bar A,s} $ and $\mathcal A_s$ to be the algebras $\mathcal F_{\bar A,s}(X)$ and  $\mathcal A_s(X)$ where $X$ is  formed by $s$ variables.
\end{definition}
\begin{lemma}\label{samid} $\mathcal A_s$ satisfies the same identities as $A$.
\end{lemma}
\begin{proof}
By construction all identities of $A$ are satisfied by $\mathcal A_s$, so we need to prove the converse and we may take a multilinear polynomial $F$ which is not a PI of $A$. Then there is a substitution of $f$ in $A$,  which we may assume to be restricted, which is non zero.

In this substitution at most $s$ variables are in the radical the others are in some matrix units of  $\bar A$.

We now take the same substitution  for matrix units in $\bar A$ and the remaining variables, which we may call $y_1,\ldots,y_k,\ k\leq s$   we substitute in $x_1,\ldots,x_k\in \mathcal A_s$.

By the universal property  the evaluation of $f$ in $A$ factors through this evaluation in $\mathcal A_s$  which therefore is different from 0.
\end{proof}
 
\begin{lemma}\label{assff}
Let $B,A=B/I$  be two finite dimensional algebras, $J$ the radical of $B$ and $I\subset J$  so $A,B$   have the same semisimple part  $\bar A$.

Assume that $A,B$ have the same nilpotency index and $A$ is fundamental, then $B$ is fundamental and with the same Kemer index as $A$.  
\end{lemma}
\begin{proof}
The assumption implies that $A,B$ have the same $t,s$ index, then the statement  follows from Theorem \ref{primoK} since by by hypothesis the  Kemer index of $A$ equals the $t,s$ index, on the other hand clearly the Kemer index of  $B$, which is less or equal than its $t,s$ index, cannot be less than the Kemer index of $A$.\end{proof}
\begin{proposition}\label{ass}
$\mathcal F_{\bar A,s} $ and $\mathcal A_s$ are fundamental algebras  with Kemer index  $t=\dim(\bar A),s$.
\end{proposition}
\begin{proof}
This follows from Theorem \ref{primoK} since by construction its Kemer index equals the $t,s$ index.\end{proof}
\begin{definition}\label{model1}
We call $\mathcal A_s$  the {\em canonical model} of $A$.
\end{definition}
\begin{definition}\label{uniFu} The algebra $\mathcal F_{\bar A,s} $ 
is the {\em universal fundamental algebra}  for $\bar A,s$.
\end{definition}
Remark that only if $m$ is sufficiently large we have that $\mathcal F_{\bar A,s}(X)$ is fundamental, since we need the existence of some fundamental algebra  quotient of  $\mathcal F_{\bar A,s}(X)$.  We claim that we have the exact condition $m\geq q-1$.

For $m=q-1$ we may take as fundamental algebra an algebra  $R$ of upper triangular matrices with $\bar A$ as semisimple part it is easily seen that such an algebra is generated by   $q-1$ elements over $\bar A$.

In fact the algebra $R$ can be described as the direct sum  $\oplus_{i\leq j}\hom(V_i,V_j)$ where $\dim V_i=n_i$.  $\bar A= \oplus_{i }\hom(V_i,V_i)$ and $\hom(V_i,V_j)$ is an irreducible module under $\hom(V_i,V_i)\oplus \hom(V_j,V_j)$ for this we may take the elements  $E_{i,i+1}$ a non zero matrix  in the corresponding block $\hom(V_i,V_{i+1})$.

Remark also that by construction the nilpotent subalgebra of  $\mathcal F_{\bar A,s}(X)$ generated by the $x_i$ is a relatively free nilpotent algebra.\smallskip

{\bf Question}  Is the $T$--ideal of identities of  $\mathcal F_{\bar A,s} $ irreducible, and how is it described?
\subsubsection{Description of $\mathcal F_{\bar A,s}(X)$} 

Let $V$ be the vector space with basis the elements $x_i$. 
In degree $h$ the algebra $\mathcal F_{\bar A,s}(X)=\mathcal F_{\bar A,s}(V)$ can be described as follows, consider 
\begin{definition}
Let $\mathcal M$ be the monoid in two generators $a,b$ with relation $a^2=a$.
\end{definition} 

Elements of this monoid correspond to words in $a,b$ in which $aa$ never appears as sub-word.  When we multiply two such words, if we have the factor $aa$ appearing we reduce it by the rule to $a$.\medskip

Now to such a word $w$ we associate a tensor product $T_w$ of $\bar A$ whenever we have an $a$ and $V$  when we have a  $b$
$$ w= abbab\mapsto T_w=\bar A\otimes V\otimes V\otimes\bar A\otimes V.$$
The multiplication of two such words $w_1w_2$ according to the previous rule, induces a multiplication $ T_{w_1}T_{w_2}\subset T_{w_1w_2}.$ We take the corresponding tensor product and if we get a factor $\bar A\otimes \bar A$ we replace it by $\bar A$ by multiplication. 

Then we see that $\mathcal F_{\bar A,s}(X)$ is the direct sum $\oplus T_w$  where $w$ runs over all the words of previous type with at most $s$  appearances of $b$. It is a  graded algebra over this monoid, truncated at degree $s$ in $b$.

As representation of  $GL(m,F)\times G=GL(V)\times G$, $\mathcal F_{\bar A,s}(V)$ in degree $h$,  is a direct  sum  of  $c_{i,j}$ spaces each isomorphic to $\bar A^{\otimes i}\otimes  V^{\otimes h}$ where $c_{i,j}$ is the number of words in $\mathcal M$ with $i$ times $a$ and $j$ times $b$. The summands correspond to the type of elements in the free product  which are monomials in $j$ elements of $V$  and $i$  elements of $\bar A$.
\begin{corollary}\label{ress}

Having fixed $s$  the $GL(V)$  invariant subspaces of  $\mathcal F_{\bar A,s}(V)$ all intersect  $\mathcal F_{\bar A,s}(V_s)$ where $V_s$   is the subspace of dimension $s$ spanned by the first $s$ variables.

\end{corollary}
\begin{proof}
This is a property of each $V^{\otimes h},\ h\leq s$.  A  $GL(V)$  invariant subspace is generated by its highest weight vectors  which in $V^{\otimes h}$ depend on the first $h$  elements of a chosen basis  of $V$ (the variables).
\end{proof}

\subsection{Some complements} 

\subsubsection{Generalized identities} 
The algebra $\mathcal F_{\bar A,s}(X)$ should be thought of as a free algebra in a suitable category, so we give the following
\begin{definition}\label{Ral2}
Given an algebra $R$, an {\em  $R$--algebra} is any associative algebra $S$ with a bimodule action of $R$  on $S$  satisfying
\begin{equation}\label{assax}
r(s_1s_2)= (r s_1)s_2,\  (s_1s_2)r=s_1(s_2 r),\ (s_1r)s_2= s_1(r s_2),\ \forall r\in R,\  s_1,s_2\in S. 
\end{equation}   \end{definition}

Given an $R$--algebra $S$  we can give to $R\oplus S$ a structure of algebra by setting
$$(r_1,s_1)(r_2,s_2)=  (r_1s_1, r_1s_2+s_1r_2+s_1s_2) $$ the axioms \eqref{assax}  are the ones necessary and sufficient to have that $R\oplus S$ is associative.\medskip

The canonical model has the following universal property, consider a nilpotent algebra $R$ satisfying the PI's of $A$, with $R^{s+1}=0$   and equipped with a  $ \bar A $ algebra structure, according to Definition \ref{Ral2}.

Then any map of $X\to R$    extends to a map $\mathcal A_s\to \bar A\oplus R$  equal to $j$ on $\bar A$.\medskip

In particular  $\mathcal F_{\bar A,s}(X)$ and $\mathcal A_s$ behave as {\em relatively free $\bar A$ algebras}.

The  endomorphisms  of  $\mathcal F_{\bar A,s}(X)$ resp. $\mathcal A_s$ which are the identity on $\bar A$ correspond to arbitrary substitutions of the variables $x_i$ with elements of the radical.

This gives rise to the notion of $T$--ideal in $\mathcal F_{\bar A,s}(X)$ or ideal of generalised identities.\smallskip

Recall that, in an algebra $R$ a {\em verbal ideal} is an ideal generated by the evaluations in $R$ of a $T$--ideal.  In other words a verbal  ideal is an ideal $I$ of $R$  minimal with respect to the property that $R/I$ satisfies some given set of polynomial identities.\smallskip

The verbal ideals of $\mathcal F_{\bar A,s} $ defining the algebras $\mathcal A_s$
are clearly $T$--ideals, but not all $T$--ideals are of this type as even the simplest examples show (cf. Example \ref{simpe}).

In particular the action of the linear group $GL(m,F)$  on the vector space with basis the  elements $x_i$ and also the automorphism group $G$ of  $\bar A$ extend to give a group $GL(m,F)\times G$ of automorphisms of  $\mathcal F_{\bar A,s}(X)$ and $\mathcal A_s$.

Thus the kernel of the quotient map $\mathcal F_{\bar A,s}(X)\to \mathcal A_s$ is stable under the group $GL(m,F)\times G$ of automorphisms.

In fact the possible ideals  of  $\mathcal F_{\bar A,s}(X)$ appearing in this way are all verbal ideals  evaluations on $\mathcal F_{\bar A,s}(X)$ of a $T$ ideal $\Gamma$  which contains the PI's of  $\bar A$. Of course, since $\bar A=\oplus_{i=1}^qM_{n_i}(F)$ this condition is that $\Gamma$    contains the PI's of  $n\times n$ matrices where $n=\max n_i$.

Not all such verbal ideals can occur but only the ones for which the Kemer index is $t,s$.  There is the further condition that   do not  decrease the nilpotency order $s+1$, this must be included if we fix $m$. If we let $m$ increase  then it will be automatically satisfied. 

In fact the previous analysis confirms the choice of  $m=s$ given in \ref{fundal} and proved in \ref{samid}.

 This may also be interpreted as fixing a Capelli identity satisfied by the algebras under consideration.
\begin{example}\label{simpe}
$\bar A=F,\ s=1$  then if $e$ is the unit of $F$  the algebra $\mathcal F_{\bar A,s}$ is 5 dimensional with basis
$$   e,\ x,\ ex,\ xe,\ exe$$
\end{example}
One can see that there are 4  verbal ideals contained in the radical, for the identities
$$[x,y],\ [x,y]z,\ z[x,y],\ z[x,y]w.$$There are 5,  $T$--ideals.  The radical is a $T$--ideal but not verbal, the corresponding verbal ideal is  $    \ ex,\ xe,\ exe$.\medskip

As for verbal ideals one may construct them as follows.
Take a multilinear polynomial $f(y_1,\ldots,y_h)$ which is also a polynomial identity of  $\bar A$, that is of $n\times n$ matrices where $n=\max n_i$.

We may consider all possible evaluations  of this polynomial in which some of the variables $y_j$ are substituted with matrix units  in $\bar A$ and the remaining variables are left unchanged but free to be evaluated in the radical.

Then we may leave out of the semisimple evaluation at most $s$ variables which we can evaluate in the elements $x_1,\ldots,x_s$.

In this way we have constructed a finite list of elements in $\mathcal F_{\bar A,s} $ and the ideal they generated is the verbal ideal associated to $f(y_1,\ldots,y_h)$.
\subsection{Moduli}
We have seen that the canonical fundamental algebras  with a given $\bar A$ and $s$  are the quotients of the finite dimensional algebra   $\mathcal F_{\bar A,s} $ modulo verbal ideals. Thus it is natural to ask if these quotients arise in algebraic families.
\begin{proposition}\label{verpr}
Let $A$ be a finite dimensional algebra, the set of ideals, respectively $T$--ideals is closed in the Grassmann variety of subspaces of codimension $h$.
\end{proposition}
\begin{proof}
Let $I$ be a subspace, the condition to be an ideal is that it should be stable under multiplication,  left and right, by elements of $A$, instead the condition of being a $T$--ideal is to be stable also under endomorphisms  of $A$.

For any linear operator $\rho:A\to A$  the set of subspaces  stable under $\rho$ is closed in the Grassmann variety hence the claim.

 \end{proof}  The condition that $I$ is verbal is that it coincides with the ideal generated by the   elements  $f(a_1,\ldots,a_h)  $ for some set of multilinear polynomials.\smallskip

Since $A$ is finite dimensional if $I$ is generated by the evaluations of some set $f$ of multilinear polynomials it is also generated by the evaluations of finitely many of them.

So every verbal ideal is in some class $V_k$ of verbal ideals which  can be generated by some multilinear polynomials $f(y_1,\ldots,y_h)$ with $h\leq k$.  In particular thus  given some sequence of possible codimensions $\underline d:=\{d_i\},\ i=1,\ldots, k$  we may consider the set $V_{h, \underline d}$ of  verbal ideals of $A$  of codimension $h$ generated by  $T$ ideals of those given codimensions.
\begin{proposition}\label{verpr1}
Let $A$ be a finite dimensional algebra, the set of   verbal ideals  $V_{h, \underline d}$  is   locally closed in the Grassmann variety of subspaces of codimension $h$.
\end{proposition} \begin{proof}
Let $M_k$ be the space of multilinear polynomials in $k$  variables  and $G_h(A)$  the Grassmann variety of codimension $h$ subspaces of $A$. In $M_k\times G_h(A)$ consider the subset  $\mathcal M_{k,h,A}$  of these pairs  $f,U$ such that all evaluations of $f$ in $A$ lie in the subspace $U$, clearly  $\mathcal M_{k,h,A}$ is closed, the projection map  $\pi_k : \mathcal M_{k,h,A}\to G_h(A)$ has the property that the fibre of $U$  is the linear subspace  of $M_k $    of the polynomials   $f $ such that all evaluations of $f$ in $A$ lie in the subspace $U$. Since the dimension of a fibre is semicontinuous we have a stratification of  $G_h(A)$  by locally closed sets where the dimension is fixed. In particular we have a locally closed subset  $G_{h, \underline d}(A)$    where the dimension of the fibre of $\pi_i$ equals $d_i$ for all $i\leq k$.

 This then intersected with the closed subset of subspaces which are ideals defines again a  locally closed set $X_{\underline d}$, on this variety we have two   vector bundles, the tautological bundle $\mathcal U_{\underline d}$ which at some point $U$ has as fibre the ideal $U$,  and the bundle $\mathcal W:=\oplus_i \mathcal W_i$  whose fibre is the direct sum of the fibres  $W_i=(\mathcal W_i)_U$ at some  ideal $U$, formed by the direct sum of the spaces of multilinear polynomials of degree $i\leq k$ which evaluated in $A$ take values which lie in $U$.
 
 In this set we have to  identify the subset of the ideals generated by the evaluations of the corresponding polynomials and we claim that this condition is open.

This condition  is given by  imposing that a certain linear map  has a maximal rank.  The map is the following  for a multilinear polynomial  $f$  in $k$ variables the span  of its image is the image of  the map  $f:A^{\otimes k}\to A$.  For a space $W$ of multilinear polynomials we have a  map $W_k\otimes A^{\otimes k}\to A$,  globally we have have a map  of vector bundles
$$\oplus_i \mathcal W_i \otimes A^i\to \mathcal U_{\underline d}. $$  For the ideal generated by these polynomials we have to add  the polynomials $xfy$ with two extra variables $x,y$ which will induce a map $\oplus_i \mathcal W_i \otimes A^{i+2}\to \mathcal U_{\underline d}. $  The verbal ideals is the open set of    $X_{\underline d}$ where this  map of vector  bundles is surjective. The condition to be surjective is then clearly open in $U$.

 \end{proof}
 Finally when we apply this to fundamental algebras we need the further restriction to take those verbal ideals of $\mathcal F_{\bar A,s} $   which do not contain the part of degree $s$, this is again an open condition.
 
 \subsubsection{ A canonical structure of Cayley--Hamilton algebra}
 We finally make a further connection with a construction from invariant theory.
 We give to $\mathcal F_{\bar A,s} $ a canonical structure of Cayley--Hamilton algebra  as follows.

Given $a\in \mathcal F_{\bar A,s} $ we set $t(a):=s tr(\bar a)$  where $\bar a\in \bar A=\oplus_{i=1}^qM_{n_i}(F)$ and we take as trace  the sum of the traces of the components which are matrices.

We claim that with this definition of trace  the elements satisfy the $(N +1)s$ Cayley--Hamilton identity  where $N=\sum_in_i$ or even the $ N  s$ Cayley--Hamilton identity if the algebra has an identity.

This follows from Formula \eqref{CH}.

We then have universal embeddings  $\mathcal F_{\bar A,s}\subset M_{(N +1)s}(U),\ \mathcal  A_s \subset  M_{(N +1)s}(U_A) $  with $U, U_A=U/I$  commutative rings over which acts the projective linear group $G:=PGL((N+1)s)$  and we have canonical isomorphisms and a commutative diagram
\begin{equation}\label{commun}\begin{CD}
\mathcal F_{\bar A,s}@>\sim>>M_{(N +1)s}(U)^G\\
@V\pi VV @V\pi VV\\
\mathcal  A_s @>\sim>>M_{(N +1)s}(U_A)^G\\
@V\pi VV @V\pi VV\\
\bar  A  @>\sim>>M_{(N +1)s}(U_{\bar A})^G
\end{CD} 
\end{equation}with the vertical maps surjective.

We can make a further reduction, by special properties of  $U_{\bar A}$  so that $\bar A$ embeds in matrices in a standard way.

This depends upon the fact that all embeddings of  $\bar A$ into $M_{(N +1)s}(F)$  which are compatible with the given trace are conjugate, this should mean that $U_{\bar A}$ is the coordinate ring of a homogeneous space and then there is a standard reduction to the subgroup.

We cannot expect to reduce the embedding as into the same matrix algebra over $F$.

\end{document}